\documentclass[reqno, 11pt]{amsart}
\usepackage{mathrsfs,graphicx,url, amssymb,enumitem}
\usepackage[usenames,dvipsnames]{xcolor}
\usepackage[colorlinks=true,linkcolor=Black,citecolor=Black, urlcolor=Black]{hyperref}

\usepackage{comment}

\newtheorem{thm}{Theorem}
\newtheorem{lem}{Lemma}

\newtheorem{prop}{Proposition}

\DeclareMathOperator \tr {tr}
\DeclareMathOperator \re {Re}
\DeclareMathOperator \im {Im}

\numberwithin{equation}{section}
\numberwithin{lem}{section}
\numberwithin{cor}{section}
\numberwithin{prop}{section}

\usepackage[letterpaper,margin=1in]{geometry}

\newcommand{\Real}{\mathbb{R}}
\newcommand{\Complex}{\mathbb{C}}
\newcommand{\Integers}{\mathbb{Z}}

\newcommand{\Natural}{\mathbb{N}}

\newcommand{\Sphere}{\mathbb{S}}
\newcommand{\pr}{\mathcal{P}}
\newcommand{\T}{T}

\newcommand{\loc}{\operatorname{loc}}
\newcommand{\dtn}{\operatorname{DtN}}
\newcommand{\restrict}{\upharpoonright}
\newcommand{\ext}{\mathcal{E}}
\newcommand{\mch}{\mathcal{H}}

\makeatletter
\@namedef{subjclassname@2020}{\textup{2020} Mathematics Subject Classification}
\makeatother

\keywords{resolvent, scattering matrix, scattering phase, Dirichlet boundary conditions, capacity}
\subjclass[2020]{35P25, 47A40, 35J25}
\title[Low energy scattering asymptotics for planar obstacles] {Low energy scattering asymptotics for planar obstacles}
\author{T. J. Christiansen and K. Datchev}
\address{Department of Mathematics, University of Missouri, Columbia, MO 65211 USA}
\email{christiansent@missouri.edu}
\address{Department of Mathematics, Purdue University, West Lafayette, IN 47907 USA}
\email{kdatchev@purdue.edu}
\begin{document}
\begin{abstract}
We compute low energy asymptotics for the resolvent of a planar obstacle, and deduce asymptotics for the corresponding scattering matrix, scattering phase, and exterior Dirichlet-to-Neumann operator. We use an identity of Vodev to relate the obstacle resolvent to the free  resolvent and an identity of Petkov and Zworski to relate the scattering matrix to the resolvent. The leading singularities are given in terms of the  obstacle's logarithmic capacity or Robin constant.  
We expect these results to hold for more general compactly supported perturbations of the Laplacian on $\mathbb R^2$, with the definition of the Robin constant suitably modified, 
under a generic assumption that the spectrum is regular at zero.
\end{abstract}
\maketitle

\section{Introduction}

\subsection{Main results} Consider the Dirichlet Laplacian $-\Delta$ on a planar exterior domain $\Omega = \mathbb R^2 \setminus \mathscr O$, where $\mathscr O \subset \mathbb R^2$ is a compact set. We study three fundamental objects in the scattering theory of $-\Delta$ on $\Omega$: the 
resolvent, the scattering matrix, and the scattering phase.
Each is a function of the frequency~$\lambda$. Our main results are uniformly convergent series expansions for all three objects near $\lambda=0$. We also deduce asymptotics for the Dirichlet-to-Neumann operator and for its lowest eigenvalue.

We begin with the resolvent $R(\lambda)$, defined for $\im \lambda>0$ to be the operator which takes $f \in L^2(\Omega)$ to the unique $u \in \mathcal D := \{u \in H^1_0(\Omega) \colon \Delta u \in L^2(\Omega)\}$ solving $(-\Delta - \lambda^2)u=f$.
For every $\chi \in C_0^\infty(\mathbb R^2),$ the product $\chi R(\lambda) \chi$ extends 
meromorphically as an operator-valued function of $\lambda$ 
to $\Lambda$, the Riemann surface of the logarithm: see Section~\ref{s:not} for more on the notation used here and below.

For our main results we assume that $\mathscr O$ is  not polar, i.e. that $C_0^\infty(\Omega)$ is not dense in $H^1(\mathbb R^2).$   For example, by Lemma \ref{l:nonpgeo}, it is enough if $\mathscr{O}$ contains a line segment.

\begin{thm}\label{t:res} Suppose $\mathscr{O}$ is not polar. Then there are operators $B_{2j,k} \colon L^2_{c}(\Omega) \to \mathcal D_{\text{loc}}$ (i.e. mapping compactly supported functions in $L^2(\Omega)$ to functions which are locally in $\mathcal D$) and a constant $a$, such that, for every $\chi \in C_0^\infty(\mathbb R^2)$, we have
\begin{equation}\label{e:resexp}\begin{split}
\chi R(\lambda) \chi &=  \sum_{j=0}^\infty   \sum_{k=-j-1}^{j} \chi B_{2j,k} \chi  \lambda^{2j} (\log \lambda-a)^k \\
&= \chi B_{0,0} \chi + \chi  B_{0,-1} \chi (\log \lambda -a)^{-1} + \chi  B_{2,1} \chi  \lambda^2 (\log \lambda -a)  + \cdots,
\end{split}\end{equation}
with the series converging absolutely in the space of bounded operators $L^2(\Omega) \to \mathcal D$,  uniformly on sectors near zero.
\end{thm}

\noindent\textbf{Remarks.} 1. Our proof also shows that if $k \ne 0$, then $B_{2j,k}$ has  finite rank. Moreover, there is a unique harmonic function $G$ in  $\mathcal D_{\text{loc}}$ such that $\log |x| - G(x)$ is bounded as $|x| \to \infty$, and 
\begin{equation}\label{e:adef}
 B_{0,-1} = \frac 1 {2\pi} G \otimes G, \quad  a = \log 2 - \gamma - C(\mathscr O) + \frac {\pi i }2,  \quad C(\mathscr O) := \lim_{|x|\to\infty} \log|x| - G(x).
\end{equation}
The quantity $C(\mathscr O)$ is important in potential theory. It is the negative of Robin's constant and as the logarithm of the logarithmic capacity: see Appendix \ref{a:nonpolar}.

\noindent 2. We used the following definition, which will recur below: Given functions $f_n$ mapping $\Lambda$ to a Banach space $\mathcal B$, we say  $\sum_n f_n(\lambda)$ \textit{converges absolutely in $\mathcal B$, uniformly on sectors near zero} if, for any $\varphi>0$, there is $\lambda_1 >0$ such that  $\sum_n \|f_n(\lambda)\|_{\mathcal B}$ converges uniformly   on $\{\lambda \in \Lambda \colon 0<|\lambda| \le \lambda_1 \text{ and } |\arg \lambda| \le \varphi\}$. Moreover, the series \eqref{e:resexp}, as well as the series \eqref{e:smatexp} below, may be freely differentiated term by term, with each resulting series having a tail which  converges absolutely in~$\mathcal B$, uniformly on sectors near zero: see Appendix~\ref{a:series}.

\noindent 3. If we used $ \lambda^{2j} (\log \lambda)^k $ instead of $ \lambda^{2j} (\log \lambda-a)^k $ in our expansion, in place of \eqref{e:resexp} we would get
\begin{equation}\label{e:noasy}
\chi R(\lambda) \chi = \sum_{j=0}^\infty   \sum_{k=-\infty}^{j} \chi \widetilde B_{2j,k} \chi\lambda^{2j} (\log \lambda)^k.
\end{equation} 
Note that terms in \eqref{e:noasy} with $j \ge 1$ have infinitely many predecessors, and hence \eqref{e:noasy} is an asymptotic expansion only as far as the terms with $j=0$. The technique of using a shift of the logarithm to reduce the number of terms comes from \cite{jensen84}.

Our second theorem concerns the scattering matrix 
 $S(\lambda)$, which is a meromorphic family of operators $L^2(\mathbb S^1) \to L^2(\mathbb S^1)$. It can be defined for $\lambda \in \Lambda$ 
 by Petkov and Zworski's formula
\begin{equation}\label{e:pz}
 S(\lambda) = I + A(\lambda), \qquad \text{where} \qquad A(\lambda) = \frac 1 {4\pi i} E(\lambda) [\Delta,\chi_1] R(\lambda) [\Delta,\chi_2] E(\overline \lambda)^*.
\end{equation}
Here  $E(\lambda)$ is the operator from $L^2(\mathbb R^2)$ to $L^2(\mathbb S^1)$ which has integral kernel $e^{-i\lambda \omega \cdot x} \chi_3(x)$, and $\chi_1,\ \chi_2,\ \chi_3$ are  radial functions in $ C_0^\infty(\mathbb R^2)$ obeying
\begin{equation}\label{e:chiprops}
\chi_1 \equiv 1 \text{ near } \overline{D_{\rho}}, \qquad \chi_i \equiv 1 \text{ near supp } \chi_{i-1}, \qquad 0 \le \chi_i \le 1,
\end{equation}
where $\rho$ is large enough that $\mathscr O \subset D_{\rho}$, and $D_{\rho} = \{x \in \mathbb R^2 \colon |x|<\rho\}$.  The formula \eqref{e:pz} comes from Theorem 4.26 of \cite{dyatlovzworski}, and is a variant of the original Proposition 2.1 of \cite{pz}. 

\begin{thm}\label{t:smat} Suppose $\mathscr{O}$ is not polar.
Then there are finite rank operators $S_{j,k} \colon L^2(\mathbb S^1) \to L^2(\mathbb S^1)$, such that
\begin{equation}\label{e:smatexp}\begin{split}
S(\lambda) 
&=   I + \frac i 2 (1 \otimes 1)(\log \lambda - a)^{-1} + S_{1,-1}\lambda (\log \lambda - a)^{-1} + \sum_{j=2}^\infty  \sum_{k=-\lfloor j/2 \rfloor-1}^{\lfloor j/2 \rfloor}   S_{j,k} \lambda^{j} (\log \lambda -a)^{k},
\end{split}\end{equation}
with the series converging absolutely in the space of trace-class operators $L^2(\mathbb S^1) \to L^2(\mathbb S^1)$,  uniformly on sectors near zero. 
\end{thm}

Our third theorem concerns the scattering phase $\sigma(\lambda)$ for $\lambda$ near zero,  defined by
\begin{equation}\label{e:scphasedef}
 \sigma(\lambda) = \frac 1 {2\pi i} \log \det S(\lambda)= \frac 1 {2\pi i} \tr \log S(\lambda).
\end{equation}

\begin{thm}\label{t:sphase}Suppose $\mathscr{O}$ is not polar. Then there are complex numbers $a_{2j,k}$ and $b_{2j,k}$ such that
\begin{equation}\label{e:sigmasum}
   \sigma(\lambda) = \frac 1 {2\pi i} \log\Big(1 + \frac {i\pi} {\log \lambda -a}\Big) +   \sum_{j=1}^\infty  \sum_{k=- \infty}^{j}    a_{2j,k} \lambda^{2j}  (\log \lambda -a)^{k},
\end{equation}
and
\begin{equation}\label{e:sigma'sum}
   \sigma'(\lambda) = \frac {-1} {2\lambda(\log \lambda -a)(\log \lambda - \bar a)} +   \sum_{j=1}^\infty  \sum_{k=- \infty}^{j}    b_{2j-1,k}  \lambda^{2j-1} (\log \lambda -a)^{k},
\end{equation}
with the series converging absolutely in $\mathbb C$, uniformly on sectors near zero. In particular, as $\lambda \to 0$ through positive real values, we have
\begin{equation}\label{e:sigma1st}
 \sigma(\lambda) = \frac 1 {\pi} \arctan\Big(\frac{\pi} {2\log (\lambda/2) + 2C(\mathscr O) + 2\gamma }\Big) + O(\lambda^2 \log \lambda),
\end{equation}
and
\begin{equation}\label{e:sigma'1st}
 \sigma'(\lambda) = \frac {-2/\lambda}{4(\log(\lambda/2) + C(\mathscr O) + \gamma)^2 +\pi^2} + O(\lambda \log \lambda).
\end{equation}
\end{thm}

\noindent \textbf{Remark.} In each of the theorems above, the assumption that $\mathscr{O}$ is not polar is necessary as well as sufficient. Indeed, if $\mathscr O$ is polar, then the expansion near $\lambda =0$ of $R(\lambda)$ is not of the form \eqref{e:resexp} but instead equals that of the free resolvent \eqref{e:r0series}. Moreover, $S(\lambda) = I$ and $\sigma(\lambda) =0$ for all $\lambda$. To see this, note that if $\mathscr O$ is polar, then $H^1_0(\Omega)$, the  form domain of the Dirichlet Laplacian on $\Omega$, is dense in $H^1(\mathbb R^2)$, the form domain of the free Laplacian on $\mathbb R^2$. Consequently the continuous extension of $R(\lambda)$ from $L^2(\Omega)$ to $L^2(\mathbb R^2)$ equals the free resolvent $R_0(\lambda)$ of Section \ref{s:resfree}.

\subsection{Background and context}

Early low frequency resolvent expansions were obtained by MacCamy \cite{maccamy}. Vainberg \cite{vai75, vai89} has very general results and many references. We focus on dimension two because of its physical importance and because the problem is harder here than in other dimensions; see for example Lemma 2.3 in \cite{lp} by Lax and Phillips for an expansion of the scattering matrix in dimension three.

Our Theorem \ref{t:res} is a variant of Theorem 2 of \cite{ww} by Weck and Witsch, of Theorem 1 of \cite{kv} by Kleinman and Vainberg, and of Theorem 1.7 of \cite{sw} by Strohmaier and Waters.

More specifically, the results of \cite{kv} and \cite{sw} cover problems which are more general than ours in many respects, but specialized to our setting they require $\partial \Omega$ to be $C^\infty$, while our assumption that $\mathscr O$ is polar is optimal. The results of \cite{ww} expand only up to $O(\lambda^2)$. Our methods are different from those in the papers mentioned above. Specifically,  \cite{kv} and \cite{sw} rely on general theory developed in \cite{vai89} and \cite{must} respectively, while our approach based on algebra of series of operators as in Vodev \cite{vo99, vo14} leads  to a direct short proof of complete resolvent asymptotics: see Section \ref{s:res}.

Our scattering phase asymptotic \eqref{e:sigma1st} improves previous results in \cite{hassellzelditch} and \cite{mcg}, and is implicit in the proof of Theorem 3.25 of \cite{sw}. To compare the results, recall that in \cite{mcg}, McGillivray computes asymptotics of the Krein spectral shift function $\xi$ defined by
\[
 \tr\Big(J g(H) J^* -  g(-\Delta)\Big) = \int_0^\infty g'(\mu)\xi(\mu)d\mu,
\]
where $J$ is the operator taking functions on $\mathbb R^2 \setminus \overline{\mathscr{O}}$ to functions on $\mathbb R^2$ by extending them by zero. By the Birman--Krein formula (see \cite{jensenkato} or Proposition 0.1 and Theorem 1.1 of \cite{chr98}), we see that $\xi(\mu) = - \sigma(\sqrt \mu)$. Putting $b =  \log 4 - 2C(\mathscr O) - 2\gamma$ and applying \eqref{e:sigma1st} gives 
\begin{align}
 \xi(\mu) &  = \frac 1 {\pi} \arctan\Big(\frac{\pi} {b - \log \mu}\Big) + O(\mu \log \mu) \label{e:xi}\\
&=  \frac {1}{-\log \mu} + \frac{- b} {(-\log \mu)^2} + \frac {b^2 - \frac {\pi^2}3}{(-\log \mu)^3} + O((\log \mu)^{-4}). \label{e:mcg}
\end{align}

Thus, \eqref{e:xi} recovers and improves \eqref{e:mcg}.  The first term of \eqref{e:mcg} was first computed in \cite{hassellzelditch}, and all three terms of \eqref{e:mcg} were computed in \cite{mcg} (note that $C(K)$ in \cite{mcg} corresponds to $2C(\mathscr O)$ here). See Figure \ref{f:comp} for a comparison of the approximations when the obstacle $\mathscr O$ is a disk of radius $\rho$, in which case the spectral shift function is given by
\begin{equation}\label{e:xihank}
 \xi_\rho(\mu) = \frac {1}{2\pi} \sum_{\ell = -\infty}^{\infty} \arg\big(H^{(2)}_\ell(\rho\sqrt{\mu})/H^{(1)}_\ell(\rho\sqrt{\mu})\big);
\end{equation}
see  \eqref{e:slh1h2}.
Note that the scaling $\xi_\rho(\mu) = \xi_{1}(\rho^2\mu)$ exhibited by \eqref{e:xihank} is also respected by our approximation in \eqref{e:xi}, but not by the individual terms of \eqref{e:mcg}.

\begin{figure}[ht]
\includegraphics[width=7cm]{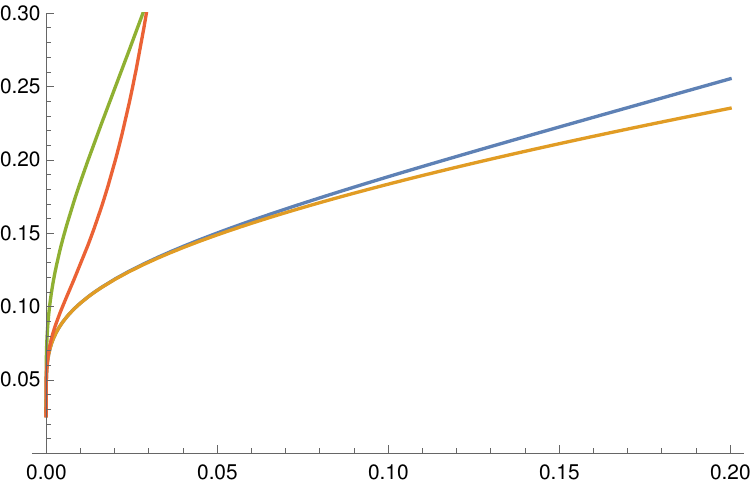} \hspace{1cm}
\includegraphics[width=7cm]{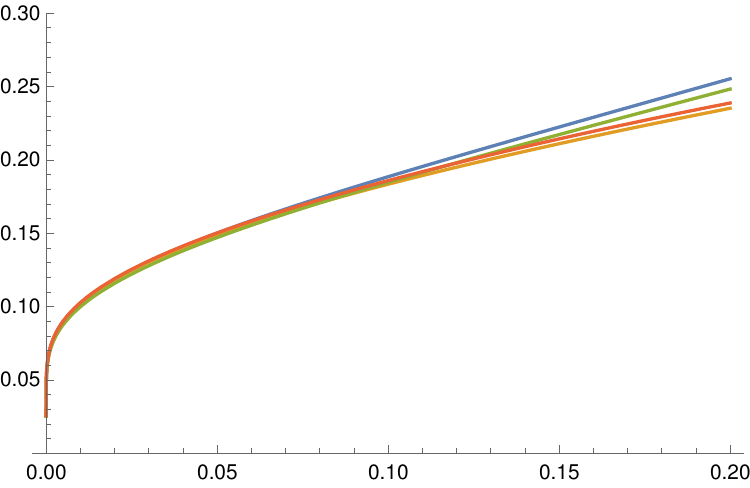}

\vspace{1cm}

\includegraphics[width=7cm]{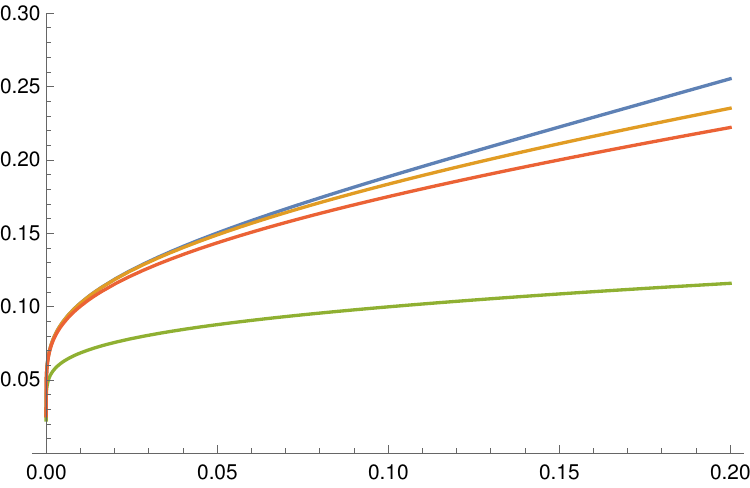} \hspace{1cm}
\includegraphics[width=7cm]{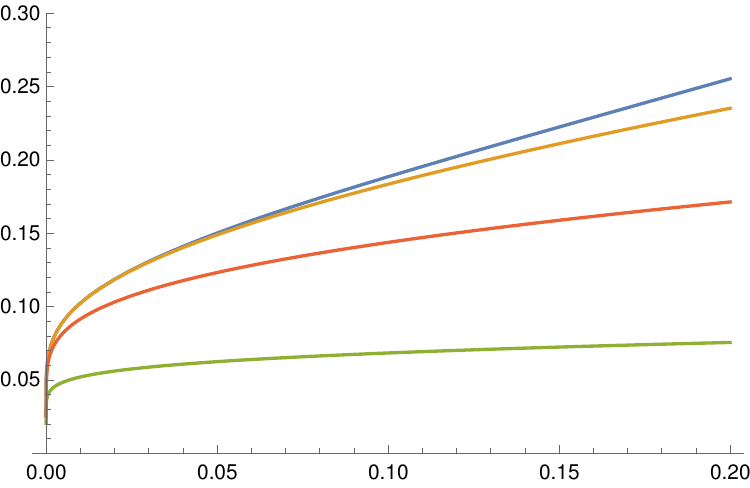}
 \caption{Graphs of the spectral shift function $\xi$ and its first three approximations, when the obstacle $\mathscr O$ is a disk of radius $\rho = $ 0.15 (top left), 1.5 (top right), 15 (bottom left), and 150 (bottom right). The horizontal axis is the dimensionless (scaling-invariant) variable $\rho \lambda = \rho \sqrt{\mu}$. The blue is the true value \eqref{e:xihank}. 
The  green is the first term from \eqref{e:mcg}, first obtained in \cite{hassellzelditch}. The red is the first three terms from \eqref{e:mcg}, first obtained in \cite{mcg}. The yellow is the leading approximation obtained in this paper, given by the first term of \eqref{e:xi}, which is implicit in the proof of Theorem 3.25 of \cite{sw}. Observe that  the yellow and blue curves are independent of the radius $\rho$, while the red and green depend on $\rho$. 
}\label{f:comp}
\end{figure}

More recently, in \cite{gmwz}, more accurate analytic and numerical investigations of the scattering phase have been conducted for both a disk and for other obstacles, including ones with different kinds of stable and unstable trapping.

\subsection{Discussion of methods and outline of proof}

The method of proof is as follows. 
In Section \ref{s:res}, we deduce the series for the resolvent $R(\lambda)$ of Theorem \ref{t:res} from the series for the free resolvent $R_0(\lambda)$, using a resolvent identity due to Vodev \cite{vo14}, and techniques in part based on Vodev's work in \cite{vo99}. 

Our techniques can be applied to quite general compactly supported
perturbations of $-\Delta $ on $\Real^2$, far beyond the specific, classical setting considered here, that of the Dirichlet problem on an exterior domain.  In a companion paper we will address the more general case.  As demonstrated
 already in \cite{bgd,jn} for the 
case of Schr\"odinger operators on $\Real^2$,  different kinds of singular behavior of the resolvent near zero can then appear.  

We deduce the series for the scattering matrix and scattering phase of Theorems \ref{t:smat} and \ref{t:sphase} in Section \ref{s:smsp}, by inserting the resolvent expansion \eqref{e:resexp} into the formulas \eqref{e:pz} and \eqref{e:scphasedef}.
Section \ref{s:dtn} contains some related results for the Dirichlet-to-Neumann operator. 

Appendix \ref{a:nonpolar} collects background on polar and nonpolar sets, Appendix \ref{a:series} contains a lemma about series with logarithmic terms, and Appendix \ref{a:disk} contains formulas for the scattering matrix and scattering phase of a disk; see also Section 5.2 of \cite{hassellzelditch} and Section 3 of \cite{gmwz}. 

The presentation of Theorems \ref{t:res}, \ref{t:smat}, and \ref{t:sphase} is self-contained, based on standard complex and functional analysis, except for a few ingredients. For the resolvent expansions we quote the series \eqref{e:r0bess} and \eqref{e:h10ser} for the free resolvent. For the scattering matrix and scattering phase  we  quote Petkov and Zworski's formula \eqref{e:pz}. Our analysis of the Dirichlet-to-Neumann operator in Section~\ref{s:dtn} requires more, namely mapping properties from \cite{gmz}, results on solutions of
boundary value problems from \cite{mclean}, and  results on analyticity of eigenvalues for which \cite{rs} is a reference.



\subsection{Notation and conventions}\label{s:not}
\begin{itemize}[leftmargin=0.5cm]
\setlength\itemsep{.15cm}
\item $C_0^\infty(U)$ is the set of functions in $C^\infty(U)$ with compact support in $U$.
\item $\mathscr{O} \subset \mathbb R^2$ is compact and $\Omega = \mathbb R^2 \setminus \mathscr O$. We say $\mathscr O$ is \textit{polar} if $C_0^\infty(\Omega)$ is dense in $H^1(\mathbb R^2)$. The Dirichlet Laplacian  is $-\Delta \colon \mathcal D \to  L^2(\Omega)$, with $\mathcal D = \{u \in H^1_0(\Omega) \colon \Delta u \in L^2(\Omega)\}$, and $\Delta := \partial_{x_1}^2 + \partial_{x_2}^2$.
\item The Dirichlet resolvent $R(\lambda)$ is defined for $\im \lambda>0$ to be the operator which takes $f \in L^2(\Omega)$ to the unique $u \in \mathcal D$ solving $(-\Delta - \lambda^2)u=f$.
For every $\chi \in C_0^\infty(\mathbb R^2)$, the product $\chi R(\lambda) \chi$ continues meromorphically from the upper half plane to $\Lambda$, the Riemann surface of the logarithm. The free resolvent $R_0(\lambda)$ is defined in the same way but with $L^2(\Omega)$ replaced by $L^2(\mathbb R^2)$ and $\mathcal D$ replaced by $H^2(\mathbb R^2)$. See Sections \ref{s:resfree} and \ref{s:vod} for a review of these facts.
\item The mapping $\Lambda \to \mathbb C$ given by $\lambda \mapsto \log \lambda$ is bijective, with the upper half plane $\im \lambda > 0$ identified with the subset of $\Lambda$ where $\arg \lambda = \im \log \lambda$ takes values in $(0,\pi)$.
\item The mapping $\lambda \mapsto \bar \lambda$ on $\Lambda$ is defined by $|\bar \lambda| = |\lambda|$ and $\arg \bar \lambda = - \arg \lambda$.
\item The product of a function $\chi$ on $\mathbb R^2$ and a function $f$ on $\Omega$ is the function $\chi f$ on $\Omega$ obtained by ignoring the values of $\chi$ on $\mathbb R^2 \setminus \Omega$.
\item $L^2_{c}(\Omega)$ is the set of functions $f \in L^2(\Omega)$ such that $\chi f =f$ for some $\chi \in C_0^\infty(\mathbb R^2)$. Other function spaces with a `c' subscript are defined analogously.
\item  $\mathcal D_{\text{loc}}$ is the set of functions $f$ on $\Omega$ such that $\chi f \in \mathcal D$ for all $\chi \in C_0^\infty(\mathbb R^2)$. Other function spaces with a `loc' subscript are defined analogously.
\item $G$ is the unique harmonic function in $\mathcal D_{\text{loc}}$ such that $\log |x| - G(x)$ is bounded as $|x| \to \infty$. We construct $G$ and prove its uniqueness in Lemma \ref{l:g}.
\item $C(\mathscr O) = \lim_{|x| \to \infty} \log |x| - G(x)$. 
\item $\gamma$ is Euler's constant, given by   $\gamma = - \Gamma'(1) = 0.577\dots $.
\item $a= \log 2 - \gamma - C(\mathscr O) + \frac{\pi i}2$.
\item $g \otimes h $ means the operator of rank one which maps $f$ to $(\int \overline{h} f) g$.
\item Petkov and Zworski's formula for the scattering matrix $S(\lambda) = I + A(\lambda)$ is given in \eqref{e:pz}. This corresponds to the definition $S(\lambda) = S_a(\lambda)S_{a,0}(\lambda)^{-1}$, where $S_a$ is the absolute scattering matrix and $S_{a,0}$ is the free absolute scattering matrix, as in Section 4.4.3 of \cite{dyatlovzworski}. Note that the convention in \cite{pz}, as in \cite{mel95}, has instead $S(\lambda) = S_{a,0}(\lambda)^{-1}S_a(\lambda)$.
\item The cutoff functions $\chi_1$, $\chi_2$, and $\chi_3$ are introduced in \eqref{e:chiprops}.
\item The scattering phase is $\sigma(\lambda) = \frac 1 {2\pi i} \log \det S(\lambda) = \frac 1 {2\pi i} \tr \log S(\lambda)$. This definition gives the same value for either convention of $S(\lambda)$.
\item $\mathbb N = \{1,\,2\,\dots\}$ and $\mathbb N_0 = \{0\} \cup \mathbb N$.
\end{itemize}

\section{The resolvent}\label{s:res}

\subsection{The free resolvent}\label{s:resfree}


Let $R_0(\lambda)$ be the free resolvent on $\mathbb R^2$, defined for $\im \lambda>0$ to be the operator which takes $f \in L^2(\mathbb R^2)$ to the unique $u \in H^2(\mathbb R^2)$ solving $(-\Delta - \lambda^2)u=f$. The integral kernel of $R_0(\lambda)$ is given in terms of Bessel functions of order $0$ by 
\begin{equation}\label{e:r0bess}
R_0(\lambda)(x,y)= \frac 1 {2\pi} K_0(-i\lambda|x-y|)=\frac{i}{4} H^{(1)}_0(\lambda|x-y|),
\end{equation}
where
$$H^{(1)}_0(\lambda) = \left( 1+\frac{2i}{\pi}\log \left(\frac{\lambda}{2}\right)  \right) J_0(\lambda) - \frac{2i}{\pi} \sum_{m=0}^\infty \psi(m+1) \frac{(-\lambda^2/4)^m}{(m!)^2},\qquad J_0(\lambda) =  \sum_{m=0}^\infty\frac{(-\lambda^2/4)^m}{(m!)^2},$$
and $\psi $ is the digamma function, i.e.  $\psi(1)=\Gamma'(1)=-\gamma$, $\psi(m+1) = \psi(m) + \frac 1m$; see e.g. \cite[Section 4.1.3]{borthwick} and \cite[Sections 2.9.3, 7.5, and 7.8.1]{olver}. Thus,
\begin{equation}\label{e:h10ser}
 H^{(1)}_0(\lambda) = \sum_{m=0}^\infty  \left(\frac{2i}{\pi}\log \left(\frac{\lambda}{2}\right) - \frac {2i}\pi \psi(m+1)+1\right)  \frac{(-\lambda^2/4)^m}{(m!)^2}.
\end{equation}

By large argument Bessel function asymptotics, (see e.g. \cite[Section 7.4.1]{olver}), for any $\lambda$ in the upper half plane we have $|R_0(\lambda)(x,y)| = O(e^{-|x-y|\im \lambda})$. Hence, for any $f \in L^2_{c}(\mathbb R^2)$ and $\lambda$ in the upper half plane,
\begin{equation}\label{e:r0exp}
 R_0(\lambda) f(x) = O(e^{-|x|\im \lambda}), \qquad \text{as } |x| \to \infty.
\end{equation}

It follows from \eqref{e:r0bess} and \eqref{e:h10ser} that  $ R_0(\lambda)(x,y)$, and hence also $R_0(\lambda)\colon L^2_{c}(\mathbb R^2)\to H^2_{\text{loc}}(\mathbb R^2)$, continue holomorphically from the upper half plane to $\Lambda$, the Riemann surface of $\log \lambda$. For each $\lambda \in \Lambda$  we write
\begin{equation}\label{e:r0series}
 R_0(\lambda) = \sum_{j=0}^\infty \sum_{k=0}^1  R_{2j,k}  \lambda^{2j}(\log \lambda)^k =  R_{01}\log \lambda + \tilde R_0(\lambda),
\end{equation}
where $\tilde R_0(\lambda)$ is defined by the equation, the $R_{2j,k}$ are operators $L^2_{c}(\mathbb R^2)\to H^2_{\text{loc}}(\mathbb R^2)$ such that $\chi R_{2j,k} \chi$ is bounded $L^2(\mathbb R^2)\to H^2(\mathbb R^2)$ for any $\chi \in C_0^\infty(\mathbb R^2)$, and the series converges in the sense that, for every $\lambda_0>0$ and $\chi \in C_0^\infty(\mathbb R^2)$, the series
\[
  \sum_{j=1}^\infty \sum_{k=0}^1 \| \chi R_{2j,k} \chi\|_{L^2(\mathbb R^2) \to H^2(\mathbb R^2)} |\lambda|^{2j}|\log \lambda|^k
\]
converges uniformly for all $\lambda \in \Lambda$ with $|\lambda| \le \lambda_0$.

 The leading coefficient is
\[
 R_{01} = - \frac{1}{2\pi} (1 \otimes 1), 
\]
where $1 \otimes 1$ means the operator mapping $\varphi$ to the constant $\int_{\mathbb R^2} \varphi$. The next term has integral kernel
\[
 R_{00}(x,y) =  - \frac 1 {2\pi} \log|x-y| + \frac{\log 2 - \gamma}{2\pi} + \frac i 4.
\]
Since $\log|x-y| = \frac 12 \log |x-y|^2 = \log |x| + \frac 12 \log(1 - \frac {2x\cdot y}{|x|^2} + \frac{|y|^2}{|x|^2})$, we obtain
\begin{equation}\label{e:r00xinf}
 R_{00} f\, (x) = \Big(- \frac 1 {2\pi} \log|x| +  \frac{\log 2 - \gamma}{2\pi} + \frac i 4\Big) \int_{\mathbb R^2} f +  \frac 1 {4\pi}\sum_{m=1}^\infty \frac 1 {m|x|^{2m}} \int  (2x\cdot y - |y|^2)^m f(y)\,dy,
\end{equation}
for $|x|$ large enough, when $f \in L^2_{c}(\mathbb R^2)$.

\subsection{The Dirichlet resolvent} 
Recall that the Dirichlet Laplacian $-\Delta$ is the unique operator $\mathcal D:= \{u \in H^1_0(\Omega) \colon \Delta u \in L^2(\Omega)\} \to L^2(\Omega)$ such that
\begin{equation}\label{e:frieddef}
 \int_{\Omega} \nabla u \cdot \nabla v  = \int_\Omega (-\Delta u) v, \qquad \text{for all } u \in \mathcal D \text{ and } v \in H^1_0(\Omega);
 \end{equation}
see the first two pages of Section 2 of Chapter 8 of \cite{taylor}, or Section 6.1.2 of \cite{borthwick}.

From the fact that $-\Delta \colon \mathcal D \to L^2(\Omega)$ is nonnegative, we know that $ \|R(i\kappa)\|_{L^2(\Omega) \to L^2(\Omega)} = \kappa^{-2}$ for all $\kappa >0$. We begin our analysis of the resolvent near $\kappa = 0$ by showing that
\begin{equation}\label{e:resapriori}
\sup_{\kappa \in (0,1]}\|\kappa^{3/2} R(i\kappa)\chi \|_{L^2(\Omega) \to H^1_0(\Omega)}  < \infty, \qquad \text{for all }\chi \in C_0^\infty(\mathbb R^2).
\end{equation}
The bound \eqref{e:resapriori} follows from $\|\chi f\|_{L^{4/3}} \le  \|\chi\|_{L^4}\|f\|_{L^2}$ (by H\"older's inequality, Theorem 189 of \cite{hlp}) and the following lemma based on a  Sobolev interpolation  inequality of Ladyzhenskaya.

\begin{lem}\label{l:resapriori}
Let $\kappa >0$,  $g \in L^2(\Omega) \cap L^{4/3}(\Omega)$, and $u = R(i\kappa)g$. Then
\begin{equation}\label{e:kinv}
\|\nabla u\|_{L^2(\Omega)}^2  +  \kappa^2  \| u\|_{L^2(\Omega)}^2 \le \frac { 1} {\kappa\sqrt 8} \|g\|_{L^{4/3}(\Omega)}^2.
\end{equation}
\end{lem}

\begin{proof}
By \eqref{e:frieddef} with $u=v$, and using H\"older's inequality, we have
\[
 \int_{\Omega} |\partial_{x_1} u|^2  +  \int_{\Omega} |\partial_{x_2} u|^2 +  \kappa^2  \int_{\Omega} |u|^2 = \int_{\Omega}  g\, \bar u \le \|g\|_{L^{4/3}} \| u\|_{L^4}.
\]
By equation (7.2) of Chapter 4 of \cite{taylor}, if $u \in H^1_0(\Omega)$ then $u \in H^1_0(\mathbb R^2)$, where we identify $u$ with its extension by $0$ from $\Omega$ to $\mathbb R^2$.
As in Supplement 2 of Chapter 1 of \cite{ladbook}, we use $|\varphi(t)| = \frac 12 |\int_{-\infty}^t \varphi' - \int_t^\infty \varphi'| \le \frac 12 \int_{\mathbb R} |\varphi'|$  to obtain
\[\begin{split}
 \int_{\mathbb R^2} |u|^4 &\le \int_{\mathbb R} \ \sup_{x_1 \in \mathbb R} |u(x_1,x_2)|^2\, dx_2  \int_{\mathbb R}\ \sup_{x_2 \in \mathbb R} |u(x_1,x_2)|^2\, dx_1 \\&\le \int_{\mathbb R^2} |u \partial_{x_1}u|\int_{\mathbb R^2} |u \partial_{x_2}u| \le \int_{\mathbb R^2}|u|^2  \Big(\int_{\mathbb R^2} |\partial_{x_1} u|^2\Big)^{1/2}\Big(\int_{\mathbb R^2} |\partial_{x_2} u|^2\Big)^{1/2}.
\end{split}\]
Hence, using also $a^{1/2}b^{1/4}c^{1/8}d^{1/8} \le \frac 12 a + \frac 14 b + \frac 18 c + \frac 18 d$ (see Theorem 9 of \cite{hlp}), we obtain
\[\begin{split}
 \int_{\Omega} |\partial_{x_1} u|^2  + & \int_{\Omega} |\partial_{x_2} u|^2 +  \kappa^2  \int_{\Omega} |u|^2 \le \|g\|_{L^{4/3}}\|u\|_{L^2}^{1/2} \|\partial_{x_1} u\|^{1/4}_{L^2} \|\partial_{x_2} u\|^{1/4}_{L^2}\\
&\le \frac{ 1}{\kappa \sqrt{32}} \Big(\int_\Omega|g|^{4/3}\Big)^{3/2} +  \frac {\kappa^2}2  \int_\Omega |u|^2 + \frac 12  \int_{\Omega} |\partial_{x_1} u|^2 + \frac 12  \int_{\Omega} |\partial_{x_2} u|^2 ,
\end{split}\]
which implies \eqref{e:kinv}.
\end{proof}

We will also need two basic facts about harmonic functions. The first is that, if $u$ is harmonic and bounded on $\{ x \in \Real^2: \; |x|>\rho\}$, then there are constants $c_0$, $c_{j,c}$, $c_{j,s}$, such that
\begin{equation}\label{eq:harmexp}
 u(r\cos \theta, r\sin \theta)=c_0+ \sum_{j=1}^\infty (c_{j,c}\cos j\theta + c_{j,s} \sin j\theta)r^{-j}, \qquad \text{for }r>\rho.
\end{equation}

The second gives the way that the assumption that $\mathscr O$ is not polar is used in our main results. 

\begin{lem}\label{l:noresat0}
If $\mathscr O$ is not polar, then the only bounded harmonic function in $\mathcal D_{\text{loc}}$  is the zero function.
\end{lem}

\begin{proof}
Let $u \in \mathcal D_{\text{loc}}$ be harmonic and bounded, and let $D_\rho = \{x \in \mathbb R^2 \colon |x|<\rho\}$. By \eqref{eq:harmexp},
\[
 \int_{\Omega \cap D_\rho} |\nabla u|^2 = \int_{\partial D_\rho} u\, \partial_r \bar u \, dS= O(\rho^{-1}), \qquad \text{ as } \rho \to \infty,
\]
which implies that $\nabla u$ is identically $0$. Because $\mathscr O$ is not polar, the only locally constant function in $\mathcal D_{\text{loc}}$ is the zero function: see the implication (2) $\Longrightarrow$ (1) of Lemma \ref{l:nonpuse}.
\end{proof}

\subsection{Vodev's identity}\label{s:vod}

Recall that $\chi_1 \in C_0^\infty(\mathbb R^2)$ is $1$ near $\mathscr{O}$. Let $z$ and $\lambda$ be in the upper half plane. To relate the resolvents $R(\lambda)$ and $R_0(\lambda)$, we start by using 
\[R(\lambda)(1-\chi_1)(-\Delta - \lambda^2)R_0(\lambda) = R(\lambda)\{(-\Delta -\lambda^2)(1-\chi_1) + [\chi_1,\Delta]\}R_0(\lambda)\]
 to write
\begin{equation}\label{e:rlam0}
 R(\lambda)(1-\chi_1) = \{1-\chi_1 - R(\lambda)[\Delta,\chi_1]\}R_0(\lambda).
\end{equation}
Similarly to \eqref{e:rlam0} we have
\begin{equation}\label{e:rlr0l}
 (1-\chi_1)R(z) = R_0(z)\{1-\chi_1 + [\Delta,\chi_1]R(z)\}.
\end{equation}
Note, for later reference, that combining \eqref{e:rlr0l} with \eqref{e:r0exp} shows that
\begin{equation}\label{e:rexp}
 f  \in L^2_{c}(\Omega) \text{ and } \im \lambda > 0 \qquad \Longrightarrow \qquad  R(z) f
 (x) = O(e^{-|x|\im \lambda}), \quad \text{as } |x| \to \infty.
\end{equation}
Inserting \eqref{e:rlr0l} and \eqref{e:rlam0} into
\[
 R(\lambda) - R(z) = (\lambda^2 - z^2)R(\lambda)R(z) = (\lambda^2 - z^2)\Big(R(\lambda)\chi_1(2-\chi_1)R(z) + R(\lambda)(1-\chi_1)^2 R(z)\Big),
\]
gives
\[
\begin{split}
R(\lambda)-R(z)= &(\lambda^2-z^2)\Big(R(\lambda) \chi_1(2-\chi_1)R(z)
\\
&+ \{ (1-\chi_1)-R(\lambda)[\Delta,\chi_1]\} R_0(\lambda)R_0(z)
\{ (1-\chi_1)+[\Delta,\chi_1]R(z)\}\Big).
\end{split}
\]
Plugging in $(\lambda^2 - z^2)R_0(\lambda)R_0(z) = R_0(\lambda)-R_0(z)$,
and introducing the notation
\begin{equation}\label{e:k1def}
 K_1 = 1 - \chi_1 +[\Delta,\chi_1] R(z),
\end{equation}
gives
\begin{equation}\label{e:vodevido}
\begin{split}
R(\lambda)-R(z)= &(\lambda^2-z^2)R(\lambda) \chi_1(2-\chi_1)R(z) + \{ 1-\chi_1-R(\lambda)[\Delta,\chi_1]\} (R_0(\lambda)-R_0(z))K_1.
\end{split}
\end{equation}
We now bring the $R(\lambda)$ terms to the left, the remaining terms to the right, and factor, obtaining
\begin{equation}\label{e:vodevid}
 R(\lambda)  (I - K(\lambda)) = F(\lambda),
\end{equation}
where
\begin{align}\label{eq:F}
 K(\lambda) &= (\lambda^2 - z^2)\chi_1(2-\chi_1)R (z)  - [\Delta, \chi_1](R_0(\lambda) - R_0(z))K_1, \nonumber\\
 F(\lambda) &= R(z) + (1-\chi_1)(R_0(\lambda) - R_0(z))K_1 .
\end{align}
Here and below we shorten formulas by using notation which displays $\lambda$-dependence but not $z$-dependence for operators other than resolvents.  The identities \eqref{e:vodevido} and \eqref{e:vodevid} are versions of Vodev's resolvent identity \cite[$(5.4)$]{vo14}.

For any $\chi \in C_0^\infty(\mathbb R^2),$ the resolvent $R(\lambda)$ continues meromorphically to $\Lambda$, the logarithmic cover of $\Complex \setminus \{0\}$. This well-known fact follows directly from Vodev's identity, as in Section~2 of \cite{cd}. Indeed, take $\chi \in C_0^\infty(\mathbb R^2)$ such that $\chi$ is $1$ near the support of $\chi_1$ and multiply \eqref{e:vodevid} on the left and right by $\chi$. That gives $\chi R(\lambda)\chi(I-K(\lambda)\chi) = \chi F(\lambda)\chi$. Observe now that $K(\lambda) \chi$ is compact $L^2(\Omega) \to L^2(\Omega)$, and $\|K(\lambda)\|_{L^2(\Omega) \to L^2(\Omega)} \to 0$ as $\lambda \to z$. Consequently, by the analytic Fredholm theorem (Theorem 4.19 of \cite{borthwick}), $\chi R(\lambda) \chi =  \chi F(\lambda)\chi(I-K(\lambda)\chi)^{-1}$ continues meromorphically from the upper half plane to $\Lambda$, the Riemann surface of $\log \lambda$.

Thus  \eqref{e:vodevido} and \eqref{e:vodevid} continue to hold for any $z$ and $\lambda$ in $\Lambda$, with $K(\lambda)$ and $K_1$  mapping 
$L^2_c(\Omega)$ to $L^2_c(\Omega)$, and
$R(\lambda)$ and $F(\lambda)$ mapping  $L^2_c(\Omega)\rightarrow \mathcal{D}_{\loc}$. 

\subsection{Resolvent expansions}

We derive our series formula for the resolvent near $\lambda=0$, which is the main result of Theorem \ref{t:res}, over the course of several lemmas, in which we establish successively more explicit formulas for $R(\lambda)$ based on Vodev's identity \eqref{e:vodevid}. The first lemma is partly based on the proof of Proposition 3.1 from \cite{vo99}:

\begin{lem}\label{l:rfk}
There is $z_0 > 0$ such that for every $z$ on the positive imaginary axis obeying $0<-iz\le z_0$ and for every $\varphi_0>0$, there is $\lambda_0>0$  such that for all $\lambda \in z$ with $|\lambda| \le \lambda_0$ and $|\arg \lambda| \le \varphi_0$, the operator $I-K(\lambda)$ is invertible $L^2_c(\Omega) \to L^2_c(\Omega)$, and 
\begin{equation}\label{e:rfk}
R(\lambda) = F(\lambda)(I-K(\lambda))^{-1}.
\end{equation}
Moreover,
\begin{equation}\label{eq:inv}
( I - K(\lambda))^{-1} = D(\lambda)\Big(I + \frac{A(\lambda) D(\lambda) }{1 - (\log \lambda - \log z)\alpha(\lambda)} \Big),
\end{equation}
where 
\[w = \frac 1 {2\pi} \Delta \chi_1, \qquad A(\lambda) = (\log \lambda - \log z) (w \otimes 1)  K_1,\] 
and
\begin{equation}\label{e:dseries}
D(\lambda) = \sum_{j=0}^\infty \sum_{k=0}^j D_{2j,k}\lambda^{2j}(\log \lambda)^k, \qquad \alpha(\lambda) = \int_{\Omega} K_1 D(\lambda)w dx=  \sum_{j=0}^\infty \sum_{k=0}^j \alpha_{2j,k}\lambda^{2j}(\log \lambda)^k,
\end{equation}
for some bounded operators $D_{2j,k}\colon L^2_c(\Omega)\to  L^2_c(\Omega)$ and complex numbers $\alpha_{2j,k}$ which depend on $z$ but not on $\lambda$. If $k\not =0$ then $D_{2j,k}$ has  finite rank. The series  converge  uniformly for all $\lambda \in \Lambda$ with $0<|\lambda| \le \lambda_0$ and $|\arg \lambda| \le \varphi_0$. Finally, we have the following variant of Vodev's identity:
 \begin{equation}
 \label{eq:vodevmod}R(\lambda)(I-A(\lambda)D(\lambda))= F(\lambda) D(\lambda).
 \end{equation}
\end{lem}

To define $A(\lambda)$ we used the notation $(w \otimes 1)f =w \int_{\Omega} f$ for $f \in L^2_c(\Omega)$, and the fact that $K_1$ maps $L^2_c(\Omega)$ into $L^2_{c}(\Omega)$.
 We understand \eqref{e:rfk}  and \eqref{eq:vodevmod} as identities of operators mapping $L^2_c(\Omega)\rightarrow \mathcal{D}_{\operatorname{loc}}$.

\begin{proof}
We define $ \tilde  K(\lambda) = K(\lambda) - A(\lambda).$ Writing
\[
 \tilde K(\lambda) = (\lambda^2-z^2) \chi_1(2-\chi_1)R(z) +[\Delta, \chi_1]\big(\tilde R_0(z) - \tilde R_0(\lambda)\big)K_1,
\]
and using the free resolvent series \eqref{e:r0series} shows that 
\[
 \tilde K(\lambda) = \sum_{j=0}^\infty \sum_{k=0}^{\min(j,1)} K_{2j,k} \lambda^{2j}(\log \lambda)^k,
\]
for some bounded operators $K_{2j,k}\colon L^2_c(\Omega)  \to L^2_c(\Omega) $ which depend on $z$ but not on $\lambda$. If $k\not =0$ then $K_{2j,k}$ has  finite rank, because, by the Hankel function series \eqref{e:h10ser}, if $k\ne 0$ then $R_{2j,k}$ has finite rank. Since
\[
 K_{00} = -z^2 \chi_1(2-\chi_1)R(z) +[\Delta, \chi_1]\big(\tilde R_0(z) - \tilde R_0(0)\big)K_1,
\]
we see, using the resolvent bounds \eqref{e:resapriori} and $\|\chi(\tilde R_0(z) - \tilde R_0(0))\chi\|_{L^2(\mathbb R^2) \to H^2(\mathbb R^2)} = O(z^2\log z)$ from \eqref{e:r0series},  that $\|K_{00}
\chi\|_{L^2 \to L^2} \to 0$ as $z_0 \to 0$.
Hence, if $|z_0|$ is small enough and $|\lambda_0|$ is small enough depending on $z$, then $I- \tilde K(\lambda)$ is invertible 
$L^2_c(\Omega) \to L^2_c(\Omega)$ with  
\begin{equation}\label{e:ikd}
 I - K(\lambda) = (I-A(\lambda)D(\lambda))D(\lambda)^{-1},  \; D(\lambda) = (I - \tilde K(\lambda))^{-1}= (I-\tilde{K}(\lambda)\chi_2)^{-1}(I+\tilde{K}(\lambda)(1-\chi_2)).
\end{equation}
Recall that  $\chi_2\in C_c^\infty(\Real^2)$ is $1$ on the support of $\chi_1$.
By Vodev's identity \eqref{e:vodevid}, it remains to show that $I-A(\lambda)D(\lambda)$ is invertible with
\begin{equation}\label{e:iad-1}
(I - A(\lambda)D(\lambda))^{-1} = I + \frac{(\log \lambda - \log z) (w \otimes 1) K_1 D(\lambda) }{1 - (\log \lambda - \log z)\alpha(\lambda)}.
\end{equation}
For this, observe that $u$ solves $u - A(\lambda)D(\lambda) u = f$ if and only if $u = f + c(\lambda) w$, where
\[
c(\lambda) w - c(\lambda)A(\lambda)D(\lambda)w = A(\lambda)D(\lambda)f.
\]
Pairing with $w$, plugging in $ A(\lambda)D(\lambda)w = (\log \lambda - \log z) \alpha (\lambda)w$, and solving for $c(\lambda)$ gives
\[\begin{split}
 c(\lambda) &= \frac{\langle A(\lambda)D(\lambda)f,w\rangle_{\mathcal H}}{(1 - (\log \lambda - \log z)\alpha(\lambda))\|w\|^2_{\mathcal H}} =  \frac{\log \lambda - \log z}{1 - (\log \lambda - \log z)\alpha(\lambda)}\int_{\mathbb R^2} K_1 D(\lambda) f.
\end{split}\]
This implies \eqref{e:iad-1} and concludes the proof.
\end{proof}

Our next lemma rewrites the formula \eqref{e:rfk} for $R(\lambda)$ as a sum of terms of rank one and terms where $\log \lambda$ does not appear in the denominator.

\begin{lem}\label{l:smsing}
Under the assumptions and notation of Lemma \ref{l:rfk}, we have
\begin{equation}\label{eq:smsing}\begin{split}
 R(\lambda) = &\frac{\log \lambda}{1-(\log \lambda-\log z)\alpha(\lambda)}\left(  \left(  \frac{  -1}{2\pi} (1-\chi_1)    +\tilde{F}(\lambda) D(\lambda) w\right) \otimes 1      \right) K_1D(\lambda)  \\ & 
 - \frac{\log z}{1-(\log \lambda-\log z)\alpha(\lambda)} \tilde{F}(\lambda)D(\lambda) (w\otimes 1)K_1D(\lambda) + 
  \tilde{F}(\lambda)D(\lambda),
\end{split}\end{equation}
where 
\begin{equation}\label{e:fseries}
 F(\lambda) = \sum_{j=0}^\infty \sum_{k=0}^1  F_{2j,k} \lambda^{2j} (\log \lambda)^k =  F_{01} \log \lambda + \tilde F(\lambda),
\end{equation}
each $F_{2j,k}$ is bounded $L^2_c(\Omega) \to \mathcal{D}_{\text{loc}}$, if $k\not =0$ then $F_{2j,k}$ has  finite rank, and
\begin{equation}\label{e:ftilde}
F_{01} = -\frac{1}{2\pi}(1-\chi_1)(1 \otimes 1)K_1.
\end{equation}
\end{lem}

\begin{proof}
The statements about the $F_{2j,k}$ follow from the free resolvent series \eqref{e:h10ser} and \eqref{e:r0series}.

Next, to simplify the contribution of $F_{01}$ in the resolvent formula \eqref{e:rfk}, we write
\[\begin{split}
(1\otimes 1) K_1 (I-K(\lambda))^{-1}& =
(1\otimes 1) K_1 D(\lambda) \left( I+ \frac{(\log \lambda -\log z)(w\otimes 1)K_1 D(\lambda) }{1-(\log \lambda-\log z) \alpha(\lambda)}\right)\\
 & = \left( 1+ \frac{(\log \lambda -\log z) \int_{\mathbb R^2} K_1 D(\lambda)w   }{1-(\log \lambda-\log z) \alpha(\lambda)}\right) (1\otimes 1) K_1 D(\lambda)  \\
 & = \left( 1+\frac{(\log \lambda -\log z)\alpha(\lambda)  }{1-(\log \lambda-\log z) \alpha(\lambda)}\right) (1\otimes 1) K_1 D(\lambda) \\
 & = \left( \frac{1 }{1-(\log \lambda-\log z) \alpha(\lambda)}\right) (1\otimes 1) K_1 D(\lambda).
\end{split}\]
Combining with the formula  \eqref{e:ftilde} for $F_{01}$ and plugging into  \eqref{e:rfk} gives
\[\begin{split}
 R(\lambda)= &- \frac 1 {2\pi} \left( \frac{\log \lambda }{1-(\log \lambda-\log z) \alpha(\lambda)}\right) ((1-\chi_1) \otimes 1) K_1 D(\lambda) + \tilde F(\lambda) (I-K(\lambda))^{-1}\\
=&\frac{\log \lambda}{1-(\log \lambda-\log z)\alpha(\lambda)}\left(  \frac{  -1}{2\pi}  ((1-\chi_1) \otimes 1)    +\tilde{F}(\lambda) D(\lambda) (w\otimes 1)      \right) K_1D(\lambda)
 \\ & +   \tilde{F}(\lambda)D(\lambda) \left(I- \frac{\log z}{1-(\log \lambda-\log z)\alpha(\lambda)} (w\otimes 1)K_1D(\lambda) \right),
\end{split}\]
which in turn gives \eqref{eq:smsing}.
\end{proof}

In our next lemma we prove that if  $\mathscr O$ is not polar, then the absolute value of the denominators in \eqref{eq:smsing} tends to infinity as $\lambda \to 0$.

\begin{lem}\label{l:a00}
 Under the assumptions and notation of Lemma \ref{l:smsing}, if $\alpha_{00} = 0$ then $\mathscr O$ is polar.
\end{lem}

For the proof of Lemma \ref{l:a00} we will use the following observation, which will also be useful later. 
Suppose there are operators $A_{2j,k}:L^2_c(\Omega)\rightarrow \mathcal{D}_{\operatorname{loc}}$, integers $K(j)\in \mathbb{N}$ and an $a\in \Complex$ so that for any $\chi\in C_0^\infty(\Real^2)$,
$\chi R(z)\chi = \sum_{j=0}^\infty \sum_{|k|\leq K(j)}\chi A_{2j,k}\chi (\log z -a)^k z^{2j}$.  Then  using $(-\Delta-z^2)R(z)=I,$ 
expanding both sides in $z$ and $\log z$  and equating like powers yields 
\begin{align}\label{eq:tecons}
-\Delta A_{2j,k}& =A_{2j-2,k} \; \text{if $(2j,k)\not = (0,0)$}\nonumber \\
-\Delta A_{0,0}& =I 
\end{align}
where we understand $A_{2j,k}=0$ if $j<0$ or if $|k|\geq K(j)$.

\begin{proof}

If $\alpha_{00}=0$, then
$|(\log \lambda - \log z)\alpha(\lambda)| \to 0$ as $\lambda \to 0$ and, using the series for $\alpha$ from \eqref{e:dseries}, we have
\[
\frac 1 {1-(\log \lambda - \log z)\alpha(\lambda)} = \sum_{m=0}^\infty (\log \lambda - \log z)^m \alpha(\lambda)^m = \sum_{j=0}^\infty\sum_{k=0}^{2j} a_{jk} \lambda^{2j}(\log \lambda)^k.
\]
Inserting this,  and the series for $D$ and $\tilde F$ from \eqref{e:dseries} and \eqref{e:fseries}, into the resolvent formula \eqref{eq:smsing} gives the resolvent series expansion 
\begin{equation}\label{e:raexp}\chi R(\lambda) \chi= \chi\left( (u_A \otimes 1) K_1 D_{00}\log \lambda + A_{00} +\sum_{j=1}^\infty \sum_{k=0}^{2j+1} A_{2j,k}\lambda^{2j}(\log \lambda)^k\right)\chi.
\end{equation}
Now by the formula for $F$ in \eqref{eq:F}, and the exponential decay estimates \eqref{e:r0exp} and \eqref{e:rexp}, we have
\begin{equation}\label{e:f00d00}
(F_{00}D_{00}w)(x)=(R_{00}K_1D_{00}w)(x)+O(e^{-|x|\im z}), 
\end{equation}
and since $\alpha_{00}=\int K_1D_{00}w=0$, 
by the asymptotics for $R_{00}$ in \eqref{e:r00xinf} we obtain 
\[u_{A}(x)= -\frac{1}{2\pi} (1-\chi_1(x)) +O(|x|^{-1}) \to - \frac 1 {2\pi}, \qquad \text{ as } |x| \to \infty.\]  Hence $u_{A}   \in \mathcal{D}_{\operatorname{loc}}$ is bounded and nontrivial.  By \eqref{eq:tecons}, we can conclude that $(\Delta u_A) \int K_1D_{00}f=0$ for all $f \in L^2_c(\Omega)$. Thus, by Lemma \ref{l:noresat0}, it is enough to show that there is $f \in L^2_c(\Omega)$ such that $\int K_1D_{00}f\ne0$. By \eqref{e:ikd}, $D_{00}\colon L^2_c(\Omega) \to L^2_c(\Omega)$ is invertible and hence it is enough to find  $f \in L^2_c(\Omega)$ such that $\int K_1f\ne0$.  By the resolvent identities \eqref{e:rlr0l} and \eqref{e:k1def} we have
\[
\int K_1 f = \int (-\Delta - z^2) (1-\chi_1)R(z) f = -z^2 \int(1-\chi_1)R(z)f,
\]
which is nonzero when $f = (-\Delta - z^2)g$ for $g \in H^2(\Omega)$ with $(1-\chi_1)g =g $ and $\int g \ne 0$.
\end{proof}

We are now ready to obtain the final form of the resolvent expansion stated in \eqref{e:resexp}.

\begin{proof}[Proof of Theorem \ref{t:res}] 
With the assumptions and notation of Lemma \ref{l:smsing}, and using the fact that $\alpha_{00} \ne 0$, we write 
\begin{equation}\label{e:1/1-r}
  \frac 1 {1-(\log \lambda - \log z)\alpha(\lambda)} =  \frac 1 {1-(\log \lambda - \log z)\alpha_{00}} \ \frac {1}{1- r(\lambda)},
\end{equation}
where
\[\begin{split}
r(\lambda)= \frac {(\log \lambda - \log z)(\alpha(\lambda)-\alpha_{00})}{1-(\log \lambda - \log z)\alpha_{00}} &= -\alpha_{00}^{-1} \left(1 - \frac{1}{1-(\log \lambda - \log z)\alpha_{00}} \right)(\alpha(\lambda)-\alpha_{00}).
\end{split}\]
Using the series for $\alpha$ from \eqref{e:dseries} to write, 
\[
 r(\lambda)^m = \sum_{j=m}^\infty \left(\sum_{k=0}^j  r_{m,2j,k} (\log \lambda)^k + \sum_{k=1}^m r_{m,2j,-k} (\log \lambda - \log z - \alpha_{00}^{-1})^{-k}\right)\lambda^{2j}, 
\]
and inserting the geometric series $\frac 1 {1-r(\lambda)} = \sum_{m=0}^\infty r(\lambda)^m$ into  \eqref{e:1/1-r}  gives
\begin{equation}\label{eq:invertscalar}
  \frac 1 {1-(\log \lambda - \log z)\alpha(\lambda)} = 
 \sum_{j=0}^\infty \left( \sum_{k=1}^{ j-1} b_{2j,k} (\log \lambda)^k + \sum_{k=0}^{j+1} b_{2j,-k} (\log \lambda - \log z - \alpha_{00}^{-1})^{-k} \right)\lambda^{2j}.
\end{equation}
Inserting this series and the series \eqref{e:dseries} and \eqref{e:fseries} for $D$ and $\tilde F$ into \eqref{eq:smsing} gives \eqref{e:resexp} with
$a=\log z +\alpha^{-1}_{00}$.  Moreover, this shows all the $B_{j,k}$ for $k\not =0$ have  finite rank, and $B_{0,-1}$ has rank at most one.
\end{proof}

We end this section with some computations concerning the resolvent expansion \eqref{e:resexp}. The first set of formulas \eqref{e:adefbis} is a restatement of \eqref{e:adef} from our main result.
The second set of formulas \eqref{e:b00b01dc2} will be used in Section \ref{s:smsp} to analyze the scattering matrix.

\begin{lem}\label{l:g}
There is a unique harmonic function $G$ in  $\mathcal D_{\text{loc}}$ such that $\log |x| - G(x)$ is bounded as $|x| \to \infty$, and 
\begin{equation}\label{e:adefbis}
 B_{0,-1} = \frac 1 {2\pi} G \otimes G, \quad  a = \log 2 - \gamma - C(\mathscr O) + \frac {\pi i }2,  \quad C(\mathscr O) := \lim_{|x|\to\infty} \log|x| - G(x).
\end{equation}
Also,
\begin{equation}\label{e:b00b01dc2}
 B_{00} \Delta \chi_2 = (1-\chi_2), \qquad B_{0,-1} \Delta \chi_2 =  G.
\end{equation}
\end{lem}

\begin{proof}
Inserting $R(i|\lambda|) = B_{00} + B_{0,-1}(\log |\lambda| - \log |z| - \alpha_{00}^{-1})^{-1} + O(\lambda^2\log \lambda)$ into $\langle R(i|\lambda|) u, v \rangle = \langle u,  R(i|\lambda|) v \rangle,$
and matching coefficients of
$(\log |\lambda| - \log |z|- \alpha_{00}^{-1})^{-1}$, gives 
\begin{equation}\label{e:b0-1pre}
 B_{0,-1} = c_B u_B \otimes u_B,
\end{equation}
for some $u_B \in \mathcal D_{\text{loc}}$ and $c_B \in \mathbb R$. From the term matching formula (\ref{eq:tecons}) we see that $-\Delta u_B=0$. 

We claim that we may choose $c_B$ and $u_B$ in such a way that $\log |x| - u_B(x)$ is bounded as $|x| \to \infty$. Then we may put $G=u_B$, as  uniqueness of $G$ follows from Lemma \ref{l:noresat0}.

To prove the claim, we use Vodev's identity in the form \eqref{eq:vodevmod}. Inserting the expansions for $D(\lambda)$, $F(\lambda)$, and $R(\lambda)$ from \eqref{e:dseries}, \eqref{e:fseries}, and \eqref{e:resexp} into \eqref{eq:vodevmod}, gives
\[\begin{split}
 (B_{00} + B_{0,-1}(\log \lambda - \log z - \alpha_{00}^{-1})^{-1})&(I + (\log z - \log \lambda)(w \otimes 1)K_1 D_{00}) \\ &= \Big(- \frac {\log \lambda} {2\pi} (1-\chi_1)(1\otimes 1) K_1 + F_{00}\Big)D_{00} +  O(\lambda^2 \log^2 \lambda).
\end{split}\]
Equating the coefficients of  $\log \lambda$ and $\lambda^0$, gives
\[
 B_{00} (w \otimes 1)K_1 D_{00} =  \frac 1 {2\pi}(1-\chi_1)(1 \otimes 1) K_1 D_{00},
\]
\[
 B_{00} + \Big(\log z\, B_{00} - B_{0,-1}  \Big)(w \otimes 1) K_1 D_{00}  = F_{00}D_{00}.
\]
The first equality implies 
\begin{equation}\label{e:b00w}
B_{00} w = \frac 1 {2\pi} (1-\chi_1). 
\end{equation}
Applying the second equality to $w$ and plugging in \eqref{e:b00w} and $\alpha_{00} = \int K_1 D_{00} w$ gives
\[
 \frac 1 {2\pi} (1-\chi_1) +  \frac {\alpha_{00}\log z } {2\pi} (1-\chi_1)  -   \alpha_{00} B_{0,-1} w = F_{00}D_{00}w,
\]
or
\begin{equation}\label{e:b0-1id}
B_{0,-1}w = - \alpha_{00}^{-1} F_{00}D_{00}w + \frac 1 {2\pi}(\alpha_{00}^{-1} + \log z) (1-\chi_1).  
\end{equation}
Using the expansion for $F_{00}D_{00}w$ in \eqref{e:f00d00} and the expansion for $R_{00}$ in \eqref{e:r00xinf} gives 
\begin{equation}\label{e:f00d00w}
 F_{00}D_{00} w(x) =  \Big(- \frac 1 {2\pi} \log|x| +  \frac{\log 2 - \gamma}{2\pi} + \frac i 4\Big) \alpha_{00}+ O(|x|^{-1}).
\end{equation}
Inserting \eqref{e:f00d00w} into \eqref{e:b0-1id}  and setting $u_B=2\pi B_{0,-1}w$ we get 
\begin{equation}\label{e:ubw}
 u_B(x) = 2\pi  B_{0,-1} w(x)=\log |x| +\alpha_{00}^{-1}+\log z -\log 2 +\gamma -\pi i/2 +O(|x|^{-1}).
\end{equation}
This completes the proof of the claim. 

Combining \eqref{e:ubw} and \eqref{eq:invertscalar} shows that the expression for $a$ in \eqref{e:adefbis} holds.
To complete the proof of \eqref{e:adefbis}, it remains to show that $c_B = 1/2\pi$. Inserting  \eqref{e:ubw} into \eqref{e:b0-1pre} gives $B_{0,-1}w = c_B G \langle w, G\rangle_{L^2} =  G/2\pi$, and so this is equivalent to showing that
\begin{equation}\label{e:wub1}
\langle w, G \rangle_{L^2} =1.
\end{equation}
Let $D_\rho = \{x \in \mathbb R^2 \colon r < \rho\}$. For $\rho$ large, the asymptotics for  harmonic functions \eqref{eq:harmexp} imply
\[\begin{split}
 \langle w, G \rangle_{L^2} &= \frac 1 {2\pi} \int_{\Omega \cap D_\rho} (\Delta \chi_1) \overline{ G} dx = \frac 1 {2\pi} \int_{\partial D_\rho} \Big( (1-\chi_1)  \partial_r \overline{ G} - \partial_r(1-\chi_1) \overline{ G}\Big) dS \\
&= \frac 1 {2\pi} \int_{\partial D_\rho} (1-\chi_1)  \partial_r \overline{ G}  dS = \frac 1 {2\pi} \int_{\partial D_\rho} \Big( r^{-1} + O(r^{-2})\Big) dS = 1 + O(\rho^{-1}).
\end{split}\]
Taking $\rho \to \infty$ completes the proof of \eqref{e:wub1}, and hence of \eqref{e:adefbis}.

It remains to prove \eqref{e:b00b01dc2}. Observe that \eqref{e:b00w} and \eqref{e:ubw} are the same as \eqref{e:b00b01dc2}, but with $\chi_2$ replaced by $\chi_1$. But \eqref{e:b00w} and \eqref{e:ubw} hold for any $\chi_1$ which is $1$ near $\mathscr O$, so they  also hold with $\chi_1$ replaced by $\chi_2$, as desired.
\end{proof}

\section{The scattering matrix and scattering phase}\label{s:smsp}
The asymptotic expansions of the scattering matrix and scattering phase follow from the asymptotic expansion of the resolvent when combined with Petkov and Zworski's formula \eqref{e:pz} which expresses the scattering matrix in terms of the resolvent.

\begin{proof}[Proof of Theorem \ref{t:smat}]
We use the series
\begin{equation}\label{e:lamser}
 E(\lambda) = \sum_{\ell=0}^\infty  \lambda^\ell E_\ell \chi_3,
\end{equation}
where
\begin{equation}\label{eq:Es}
E_\ell f(\omega) :=  \frac {1}{\ell!} \int_{\mathbb R^2} (-i \omega \cdot x)^\ell f(x)dx =  \sum_{m=0}^\ell \frac {(-i)^\ell \omega_1^m \omega_2^{\ell-m} }{m!(\ell-m)!} \int_{\mathbb R^2} x_1^m x_2^{\ell-m}f(x)dx.
\end{equation}
Since a rank one operator $\varphi \otimes \psi \colon \mathcal H \to \mathcal H'$ has trace norm  $\|\varphi\|_{\mathcal H'}\|\psi\|_{\mathcal H}$, we obtain
\begin{equation}\label{e:echi3tr}
 \|E_\ell\chi_3 \|_{\text{tr}} \le  \sum_{m=0}^\ell \frac {\|\omega_1^m \omega_2^{\ell-m}\|_{L^2(\mathbb S^1)} \|x_1^m x_2^{\ell-m} \chi_3(x)\|_{L^2(\mathbb R^2)}}{m!(\ell-m)!}  \le \sum_{m=0}^\ell \frac {\sqrt{2} \pi \rho^{\ell+1}}{ m!(\ell-m)!} = \sqrt{2}\pi \frac{2^\ell \rho^{\ell +1}}{\ell!},
\end{equation}
where $\rho$ is chosen large enough that $|x| \le \rho$ on the support of $\chi_3$, and we used the fact that $\max|\chi_3|=1$.
Plugging in the series \eqref{e:resexp} for $R(\lambda)$, and the series \eqref{e:lamser} for $E(\lambda)$,
into Petkov and Zworski's formula for the scattering matrix \eqref{e:pz}, gives
\begin{equation}\label{e:aser}
 A(\lambda) = \frac {1} {4\pi i} \sum_{\ell=0}^\infty \sum_{\ell'=0}^\infty \sum_{j=0}^\infty \sum_{k=-j -1}^{j } E_{\ell}  [\Delta,\chi_1]B_{2j,k} [\Delta,\chi_2] E_{\ell'}^* \lambda^{2j+\ell+\ell'}(\log \lambda - a)^k .
\end{equation}
The series \eqref{e:aser} is absolutely convergent in the sense that
\[
 \sum_{\ell=0}^\infty \sum_{\ell'=0}^\infty \sum_{j=0}^\infty \sum_{k= -j -1}^{j } \|E_{\ell}  [\Delta,\chi_1]B_{2j,k} [\Delta,\chi_2] E_{\ell'}^*\|_{\text{tr}} |\lambda|^{2j+\ell+\ell'} |\log \lambda - a|^k,
\]
converges for  $\lambda$ on the positive imaginary axis with $|\lambda|$ small enough, because it is the product of the convergent series for $ E(\lambda)   \colon L^2(\mathbb R^2) \to L^2(\mathbb S^1)$, $ [\Delta,\chi_1] R(\lambda) \chi_3 \colon L^2(\Omega) \to L^2(\mathbb R^2)$, and $[\Delta,\chi_2] E(\overline{\lambda})^* \colon L^2(S^1) \to L^2(\Omega)$, with the series for $E(\lambda) $  being convergent in the trace norm by~\eqref{e:echi3tr}.

Hence, by Lemma \ref{l:appser}, the series \eqref{e:aser} is absolutely convergent in the space of trace-class operators $L^2(\mathbb S^1) \to L^2(\mathbb S^1)$, uniformly on sectors near zero.  Since the $E_\ell$  have finite rank, the $S_{j,k}$ also have finite rank.

To complete the proof of Theorem \ref{t:smat}, it remains to simplify the leading order terms in \eqref{e:aser}.

We first show that all terms of \eqref{e:aser} with $j=k=0$ and $\ell+ \ell' \le 1$ simplify to $0$. By \eqref{e:b00b01dc2}, we have $B_{0,0}[\Delta,\chi_2]  E_0^*=(1-\chi_2)\otimes 1$, and hence $ [\Delta,\chi_1]  B_{0,0}[\Delta,\chi_2]  E_0^*=0$. Hence all terms of \eqref{e:aser} with $j=k=\ell'=0$ simplify to $0$. To prove that the term with $j=k=\ell=0$ and $\ell'=1$ simplifies to $0$ too, we observe that since $B_{0,0}^* = B_{0,0}$, we have $E_0 [\Delta,\chi_1] B_{0,0} = 1 \otimes(1-\chi_1)$, and so
\[\begin{split}
  E_0  [\Delta,\chi_1]B_{0,0} [\Delta,\chi_2] E_1^* f &= 1 \otimes(1-\chi_1) [\Delta,\chi_2] E_1^* f  =  i\int  [\Delta,\chi_2] \int x \cdot \omega f(\omega)dS dx,\\
\end{split}\]
while integrating by parts shows
\[
 \int [\Delta,\chi_2] x_j = \int (\Delta \chi_2)x_j + 2 \partial_{x_j} \chi_2 = 0.
\]

Second, we simplify the term with $j=\ell=\ell'=0$ and $k = -1$. By \eqref{e:b00b01dc2} we have
\[
  E_{0}  [\Delta,\chi_1]B_{0,-1} [\Delta,\chi_2] E_{0}^* = \Big(1 \otimes \chi_3(x)\Big)  [\Delta,\chi_1]\Big(G(x)\otimes 1\Big).
\]
Moreover, for $\rho$ large enough we have
\[- \int_{\mathbb R^2} [\Delta,\chi_1] G =  \int_{B(0,\rho)} \Delta(1-\chi_1) G =  \int_{r=\rho} \partial_r  G = 2\pi  + O(\rho^{-1}) \to  2\pi, \qquad \text{as } \rho \to \infty,\]
so we get
\[
\Big(1 \otimes \chi_3(x)\Big)  [\Delta,\chi_1]\Big(G(x)\otimes 1\Big) = -2 \pi (1 \otimes 1),
\]
which completes the calculation of $S_{0,-1}$. 
\end{proof}

We introduce the following notation to help with the proof of Theorem \ref{t:sphase}.  For $t\in [0,\infty)$ we denote by  $\{t\}$ the fractional part of $t$. 
Thus, for $\ell \in \mathbb N_0$, $\{\ell/2\} = 0$ if $\ell$ is even and $\{\ell/2\} = 1/2$ if $\ell$ is odd.
We define two orthogonal projection operators acting on $L^2(\Sphere^1)$, $\pr_0$
and $\pr_{1/2}$.  The projection $\pr_0$ projects onto the span of $\{ e^{2 i n \theta},\ n\in \Integers\}$ and $\pr_{1/2}$ projects onto 
the span of $\{ e^{(2 n+1)i \theta},\ n\in \Integers\}$.   Hence $\pr_0\pr_{1/2}=0=\pr_{1/2}\pr_0$, and $\pr_0+\pr_{1/2}=I$.  
The following elementary lemma gives an indication of how we will use these projections.
\begin{lem}\label{l:oddtrace}
Let $m\in \Natural$.  For $j=1,...,m$, let $\ell_j, \; \ell'_j\in \Natural_0$ and $A_j:L^2(\Sphere^1)\to L^2(\Sphere^1)$ be trace class.  If 
$\sum_{j=1}^m(\ell_j+\ell_j')$ is odd, then 
$$\tr\left( \pr_{\{\ell_1/2\}}A_1\pr_{\{\ell_1'/2\}}\pr_{\{\ell_2/2\}}A_2\pr_{\{\ell_2'/2\}}\cdot \cdot \cdot \pr_{\{\ell_m/2\}}A_m\pr_{\{\ell_m'/2\}}\right)=0.$$
\end{lem}
\begin{proof}
Since $\sum_{j=1}^m(\ell_j+\ell_j')$ is odd, there must be at least one $j_0$ with $1\leq j_0<m$ so that $\ell'_{j_0}$ and $\ell_{j_0+1}$ have 
opposite parities, or $\ell_1$ and $\ell'_m$ must have opposite parities (or both).  By the cyclicity of the trace, for the purposes of computing the 
trace we can assume the latter holds.  But in this case since $\pr_{1/2}\pr_0=0=\pr_0\pr_{1/2}$, noting that for a trace class $T:L^2(\Sphere^1)\rightarrow L^2(\Sphere^1)$, $\tr( \pr_0 T \pr_{1/2})= \tr(  \pr_{1/2}\pr_0 T \pr_{1/2})=0$ and $\tr ( \pr_{1/2}T\pr_0)=\tr(\pr_0\pr_{1/2}T\pr_{0})=0$ proves the lemma.
\end{proof}

\begin{proof}[Proof of Theorem \ref{t:sphase}]
By the series \eqref{e:smatexp} for $S(\lambda) = I + A(\lambda)$ we see that $\|A(\lambda)\|_{\text{tr}} = O(1/\log \lambda)$, and so (by Theorem 1 of Section 4.3 of \cite{knopp} on substituting a convergent series into a power series) we have the following convergent series for the scattering phase:
 \begin{equation}\label{e:sigmasuma}
  \sigma(\lambda) = \frac 1 {2\pi i}\tr \log(I + A(\lambda)) =  \frac {-1}{2\pi i} \sum_{\ell=1}^\infty \frac {(-1)^\ell}{\ell} \tr A(\lambda)^\ell.
\end{equation}
Using the expression \eqref{eq:Es} for $E_\ell$, we see that 
\begin{equation}\label{eq:prE}
\pr_{\{\ell/2\}}E_{\ell}=E_\ell\; \text{ and } \;\pr_{\{(\ell+1)/2\}}E_{\ell}=0.
\end{equation}
In order to get an odd power of $\lambda$ in the expansion \eqref{e:aser} we see that exactly one of $\ell$ or $\ell'$ must be odd.  This implies
by \eqref{eq:prE} that the coefficients of $ \lambda^{2n+1}(\log \lambda -a)^k$ in the expansion of $A=S-I$ can be written in the form 
\begin{equation}\label{eq:Sodd}
S_{2n+1,k}=\pr_0 \tilde{S}_{2n+1,k,eo}\pr_{1/2}+\pr_{1/2} \tilde{S}_{2n+1,k,oe}\pr_{0}
\end{equation}
 for some trace class operators  $\tilde{S}_{2n+1,k,eo}$ and $\tilde{S}_{2n+1,k,oe}$.  Likewise, the coefficients of $ \lambda^{2n}(\log \lambda -a)^k$ in the expansion of $S$ can be written 
 \begin{equation}\label{eq:Seven}
 S_{2n,k}=\pr_0 \tilde{S}_{2n,k,ee}\pr_{0}+\pr_{1/2} \tilde{S}_{2n,k,oo}\pr_{1/2},\; \qquad \text{when } (n,k)\not =(0,0),
 \end{equation} for some trace class operators $S_{2n,k,ee}$ and $S_{2n,k,oo}$.

Using \eqref{eq:Sodd}, \eqref{eq:Seven} and Lemma \ref{l:oddtrace} yields
\[
 \tr A(\lambda) = \frac {i \pi} {\log \lambda -a} + \sum_{j=1}^\infty  \sum_{k=-j-1}^j
   \tr S_{2j,k}  (\log \lambda -a)^{k}\lambda^{2j}. 
\]
Note there are no odd powers of $\lambda$.
Moreover, 
 there are constants $a_{j,k,l}$ such that 
\begin{equation}\label{e:tral}\begin{split}
 \tr A(\lambda)^\ell 
&= \Big( \frac {i\pi} {\log \lambda -a} \Big)^\ell  +  \sum_{j=1}^\infty  \sum_{k=-j-\ell}^j
a_{j,k,\ell}  (\log \lambda -a)^{k}\lambda^{2j}.
\end{split}\end{equation}
Again, the absence of odd powers of $\lambda$ comes from applying Lemma \ref{l:oddtrace}, \eqref{eq:Sodd} and \eqref{eq:Seven}.
Substituting \eqref{e:tral} into \eqref{e:sigmasuma} gives
\begin{equation}\label{e:sigmasumsum}
 \sigma(\lambda) = \frac {-1}{2\pi i} \sum_{\ell=1}^\infty \frac {(-1)^\ell}{\ell} \Bigg(\Big( \frac {i\pi} {\log \lambda -a} \Big)^\ell +   \sum_{j=1}^\infty  \sum_{k=-j-\ell}^j
 a_{j,k,\ell}  (\log \lambda -a)^{k}\lambda^{2j} \Bigg),
\end{equation}
and combining with
\[
\sum_{\ell=1}^\infty \frac {(-1)^\ell}{\ell} \Big( \frac {i\pi} {\log \lambda -a} \Big)^\ell  = - \log\Big(1 + \frac {i\pi} {\log \lambda -a}\Big),
\]
we obtain \eqref{e:sigmasum}. Then \eqref{e:sigma1st} follows from \eqref{e:sigmasum} together with
\begin{equation}\label{e:lnargarctan}
 \log\Big(1 + \frac {i\pi} {\log \lambda -a}\Big) = i  \arg\Big(1 + \frac {i\pi} {\log \lambda -a}\Big) = 2 i \arctan\Big(\frac{\pi} {2\log (\lambda/2) + 2C(\mathscr O) + 2\gamma }\Big).
\end{equation}
Absolute convergence of the series for $\sigma(\lambda)$ and $\sigma'(\lambda)$, uniformly on sectors near zero, again follows from Lemma~\ref{l:appser}. Differentiating \eqref{e:lnargarctan}  gives \eqref{e:sigma'1st}.
\end{proof}

\section{The Dirichlet-to-Neumann operator}\label{s:dtn}

\noindent In this section we show that the Dirichlet-to-Neumann operator for the exterior Helmholtz equation 
\begin{equation}\label{eq:udef}\begin{split}
(-\Delta -\lambda^2)u=0, \qquad   &\text{in } \Omega, \\
u=f, \qquad &\text{on } \partial \Omega,
\end{split}\end{equation}
 has an expansion near $\lambda = 0$ very much like that of the resolvent in Theorem \ref{t:res}.  In fact, the expansion follows easily from our Theorem \ref{t:res}.  We 
use this to answer  a question raised by D. Grebenkov regarding the lowest eigenvalue of the Dirichlet-to-Neumann operator near~$\lambda = 0$ \cite[page 11]{gr22}. See \cite[Chapter 4]{mclean}, \cite[Chapter 3]{coltonkress}, and  \cite[Chapters 7 and 9]{taylor} for textbook introductions to the Dirichlet-to-Neumann operator. 
The papers \cite{ArEl,bsw} contain results on exterior Dirichlet-to-Neumann operators and some references to further results on the subject.

In this section for simplicity we assume that $\partial \mathscr{O} = \partial\Omega$ is smooth, without boundary, and  $\Omega = \Real^2\setminus \mathscr O$ is connected.

 Let $\im \lambda>0$, and define $\dtn(\lambda)$, the Dirichlet-to-Neumann operator on
$\Omega$,
 to be the operator that maps $H^{1/2}(\partial \Omega)\ni f \mapsto \partial_\nu u \in H^{-1/2}(\partial \Omega)$
where $u\in L^2(\Omega)$ satisfies \eqref{eq:udef} and $\partial_\nu$ is the outward (with respect to $\Omega$) pointing unit normal.  Here we use 
an extended notion of the normal derivative as described in \cite[ Lemma 4.3]{mclean}. Then $\dtn(\lambda)$ 
has a meromorphic continuation to $\Lambda$. We prove this well-known fact in the course of the proof of our next theorem, which is the analog of Theorems \ref{t:res}, \ref{t:smat}, and \ref{t:sphase} for the Dirichlet-to-Neumann operator.

\begin{thm} \label{t:dtn} Let $\partial \Omega$ be smooth and $\Omega$ be connected. 
  There are operators $\T_{2j,k}: H^{1/2}(\partial \Omega)
\rightarrow H^{-1/2} (\partial \Omega)$ such that 
\begin{equation}\label{e:dtnexp}\dtn(\lambda )= \sum_{j=0}^\infty   \sum_{k=-j-1}^{j}  \T_{2j,k}  \lambda^{2j} (\log \lambda  -a)^k,
\end{equation}
with $a$ as in \eqref{e:adef}, and with the series converging absolutely in the space of bounded operators $ H^{1/2}(\partial \Omega)
\rightarrow H^{-1/2} (\partial \Omega)$,  uniformly on sectors near zero.
\end{thm}
In fact, our first proof shows that if $(j,k)\not = (0,0)$, then $T_{2j,k}:H^{1/2}(\partial \Omega) \rightarrow H^{1/2}(\partial \Omega) $.

We remark that for $\lambda >0$ the function $u$ used in the definition of the Dirichlet-to-Neumann operator can be uniquely determined by requiring it to satisfy a Sommerfeld radiation condition rather than be in $L^2(\Omega)$; this is done in \cite{bsw}, for example.

\begin{proof}
We begin by describing a construction of the unique
$u=u(x,\lambda)\in L^2(\Omega)$ satisfying \eqref{eq:udef}
 for $\im \lambda >0$ and $f \in H^{1/2}(\partial \Omega)$.   Choose
 $\rho>0$ so that $\mathscr{O}\subset D_\rho= \{x\in \Real^2:\; |x|<\rho\}$.  By 
 \cite[Theorem 4.10]{mclean} there is a unique $\tilde{u}\in H^1(\Omega \cap D_\rho)$ satisfying
 \begin{align*}
 -\Delta \tilde{u}=0\; & \text{on} \; \Omega\cap D_\rho\\
 \tilde{u}\restrict_{\partial \Omega}=f\\
 \tilde{u}\restrict_{\partial D_\rho}=0.
 \end{align*}
 Denote by $\mathcal{U} $ the mapping $f\mapsto \tilde{u}$; by \cite[Theorem 4.10]{mclean} this is a continuous map $\mathcal{U}: H^{1/2}(\partial \Omega)\rightarrow H^{1}(\Omega \cap D_\rho)$.  Choose $\chi \in C_0^\infty(D_\rho)$ so that $\chi$ is one in a neighborhood of $\overline{\mathscr{O}}$, and set 
 \begin{equation}\label{eq:u}
 u=\chi \mathcal{U}f- R(\lambda)(-\Delta-\lambda^2) (\chi \mathcal{U}f) = \chi \mathcal{U}f+R(\lambda)([\Delta,\chi] \mathcal{U}f+\lambda^2\chi\mathcal{U}f) .
 \end{equation}
 Since $([\Delta,\chi] \mathcal{U}f+\lambda^2 \chi\mathcal{U}f)\in L^2_c(\Omega)$, $R(\lambda)([\Delta,\chi] \mathcal{U}f+\lambda^2\chi\mathcal{U}f)$
 has a meromorphic continuation (as an element of $H^2_{\loc}(\Omega)$) to $\Lambda$, as does 
 $\partial_\nu R(\lambda)([\Delta,\chi] \mathcal{U}f+\lambda^2\chi\mathcal{U}f)\in H^{1/2}(\partial \Omega)$.  This shows $\dtn(\lambda)$ has a meromorphic
 extension to $\Lambda$.  Moreover, the expansion \eqref{e:dtnexp} for $\dtn(\lambda)$ follows from the expansion \eqref{e:resexp} for $R(\lambda)$
 and the expression (\ref{eq:u}).
\end{proof}

We sketch an alternate proof of Theorem \ref{t:dtn}.  This second proof is more similar in approach to our proof of Proposition \ref{p:dtnev}.  Moreover,
it addresses the mapping properties of $R(\lambda)$ and $B_{2j,k}$ acting on distributions with less regularity than $L^2$, which is of independent interest.
\begin{proof}
Let $\ext:H^{1/2}(\partial \Omega)\rightarrow H^1_{\loc}(\Omega)$ be an extension operator, so that
$\ext f \restrict_{\partial \Omega}= f$: see  \cite[Theorem 1.5.1.2]{grisvard}.   This extension is not uniquely determined, but any such extension will do.  
 Let $\chi\in C_0^\infty(\Real^2)$ be $1$ in a neighborhood of $\Omega$.   Then $(-\Delta -\lambda^2) (\chi \ext f)\in H^{-1}(\Omega):=(H^1_0(\Omega))^*$ and has 
 compact support.  
 
 For $\im \lambda>0$ the function $u= \chi \ext f+R(\lambda)(\Delta+\lambda^2)\chi \ext f$ satisfies \eqref{eq:udef} and $u\in L^2(\Omega)$.  This is straightforward if $f\in H^{3/2}(\Omega)$, but we need
 some argument that this also holds for $f\in H^{1/2}(\partial \Omega)$.  In addition, we will need that the expansion \eqref{e:resexp}  holds as a map
 $H^{-1}_c(\Omega)\rightarrow (H^1_0(\Omega))_{\loc}$.
 Suppose $-i\im z >0$ and $g\in H^{-1}(\Omega)$.  We consider a weak formulation of  the problem $(-\Delta -z^2)v=g$ via a bilinear form.
 By the Riesz representation theorem (\cite[Theorem 2.28]{borthwick} or
 the Lax-Milgram theorem
 (e.g. \cite[Section 6.2.1]{evans}) there is a unique $v\in H^1_0(\Omega)$ so that for all $w\in H^1_0(\Omega)$,  
 $\langle \nabla v,\nabla w\rangle -z^2 \langle v,w\rangle 
 = \langle g, w\rangle$, where the last pairing is the dual pairing of $H^{-1}(\Omega)$ and $H^1_0(\Omega)$.   Moreover, $v$ depends continuously on $g$.
 Hence $R(z): H^{-1}(\Omega)\rightarrow H^1_0(\Omega)$ continuously when $-iz>0$ (with the norm depending on $z$).
    Now we use the notation of the proof of Theorem \ref{t:res}.
 Since $R_0(z)$ commutes with the Laplacian, for $\im z>0$ $R_0(z):H^{-1}(\Real^2)\rightarrow H^1(\Real^2)$ and with $R_{2j,k}$ from \eqref{e:r0series},
 $R_{2j,k}:H_c^{-1}(\Real^2)\rightarrow H^1_{\loc}(\Real^2).$
  This and the mapping properties
 of $R(z)$ show 
 that $F_{2j,k}: H^{-1}_c(\Omega)\rightarrow (H^1_0(\Omega))_{\loc}$.  Then since $ D_{2j,k}: H^{-1}_c(\Omega)\rightarrow  H^{-1}_c(\Omega)$, inspection of our proof of Theorem \ref{t:res} shows that $B_{2j,k}: H^{-1}_c(\Omega)\rightarrow ( H^1_0(\Omega))_{\loc}$.
 
 Thus we have, for $\im \lambda>0$
 \begin{equation}
 \label{eq:utwo}
 u = \chi \ext f + R(\lambda) (\Delta +\lambda^2) (\chi \ext f)\in H^1(\Omega).
 \end{equation} Therefore $u$ has a meromorphic continuation to $\Lambda$ and has an expansion near $0$ in powers of $\lambda$ and $(\log \lambda -a)$
 just as the resolvent does,  \eqref{e:resexp}.  Hence  $\partial_\nu u=\dtn (\lambda)u$ has a meromorphic continuation to $\Lambda$ and  the expansion \eqref{e:dtnexp} for $\dtn(\lambda)$ follows from the expansion \eqref{e:resexp} for $R(\lambda)$ and \eqref{eq:utwo}.
 \end{proof}


For the convenience of the reader, we include brief proofs of two variational formulas for eigenvalues, which we shall use below. The first, \eqref{eq:hv1}, is known as Hadamard's variational formula and as the Feynman--Hellmann Theorem (see \cite[Theorem 1.4.7]{simon}). Both \eqref{eq:hv1} and \eqref{eq:hv2} are essentially special cases of equation (2.36) of Chapter II of \cite{kato}, which generalizes well-known perturbation theory formulas from quantum mechanics as in equations (7.9) and (7.15) of \cite{griffiths}.

\begin{lem}[Variational formulas]
Let $\mathcal{H}$ be a 
Hilbert 
space with inner product $\langle \bullet, \bullet \rangle_{\mch}$, and let $A(\tau):\mathcal{H}\rightarrow \mathcal{H}$ be a (possibly unbounded) self-adjoint linear operator depending
in a $C^1$ fashion  on 
$\tau\in  (\alpha,\beta)\subset \Real$.   Let $\phi_1(\tau)$ be an eigenfunction of $A(\tau)$, which has
$\| \phi_1(\tau)\|_{\mch}=1$ and which depends in a $C^1$ fashion on $\tau\in (\alpha, \beta)$.  Suppose $A(\tau)\phi_1(\tau)=\sigma_1(\tau)\phi_1(\tau)$.  Then $\sigma_1\in C^1(\alpha,\beta), $ and
\begin{equation}\label{eq:hv1}\partial_\tau \sigma_1=\langle (\partial _\tau A)\phi_1,\phi_1\rangle _{\mch}.
\end{equation}
For the second variational formula, suppose in addition that $A(\tau)$ and $\phi_1(\tau)$ are $C^2$, that $\mch$ is separable, and that for $\tau\in (\alpha, \beta)$
 $A(\tau)$ has a complete orthonormal set of 
eigenfunctions $\{\phi_j\}_{j=1}^\infty=\{\phi_j(\tau)\}_{j=1}^\infty$  with $A(\tau) \phi_j(\tau)=\sigma_j(\tau)\phi_j(\tau)$.  Moreover, assume that $\sigma_j\not = \sigma_1$ if $j\not =1$.
Then $\sigma_1\in C^2(\alpha,\beta)$, and 
\begin{equation}\label{eq:hv2}
\partial_\tau^2 \sigma_{1}= \langle (\partial_\tau^2  A )\phi_1, \phi_1\rangle_{\mch}+ 2 \sum_{j\not = 1} \frac{1}{\sigma_1-\sigma_j}\left| \langle (\partial_\tau A )\phi_1,\phi_j \rangle_{\mch}\right|^2.
\end{equation}
\end{lem}
\begin{proof}
 Since $\sigma_1=\langle \phi_1,\phi_1\rangle$ and the right hand side is differentiable on $(\alpha,\beta)$, so is $\sigma_1$.
Then
$$\partial_\tau \sigma_1=\partial_\tau \langle   A \phi_1, \phi_1\rangle_{\mch}= \langle (\partial _\tau A)\phi_1,\phi_1\rangle_{\mch} + \langle  A \partial_\tau \phi_1,\phi_1\rangle_{\mch} + \langle A\phi_1,\partial_\tau \phi_1\rangle_{\mch}.$$
By the self-adjointness of $A$, this is 
\begin{align}
\partial_\tau \sigma_1& =  \langle (\partial _\tau A)\phi_1,\phi_1\rangle _{\mch}+ \langle   \partial_\tau \phi_1,A\phi_1\rangle_{\mch} +\langle  A\phi_1,\partial_\tau \phi_1\rangle_{\mch}\nonumber \\ & =
\langle (\partial _\tau A)\phi_1,\phi_1\rangle _{\mch}+ \sigma_1\left( \langle   \partial_\tau \phi_1,\phi_1\rangle_{\mch} +\langle  \phi_1,\partial_\tau \phi_1\rangle_{\mch}\right)
= \langle (\partial _\tau A)\phi_1,\phi_1\rangle _{\mch}
\end{align}
since $\partial_{\tau}\langle \phi_1,\phi_1 \rangle_{\mch}=0$.  This yields the first variational formula, \eqref{eq:hv1}.

For the second variational formula 
we use 
$$(\partial_\tau A)\phi_1+ A\partial_\tau \phi_1 = \partial_\tau \sigma_1 \phi_1 + \sigma_1 \partial_\tau \phi_1.$$
Taking the inner product with $\phi_j$ when $j\not = 1$ yields, after some simplification and rearrangement,
\begin{equation}\label{eq:ip1}
 \langle \partial_\tau \phi_1,\phi_j\rangle_{\mch}=\frac{1}{\sigma_1-\sigma_j}\langle (\partial_\tau A)\phi_1,\phi_j\rangle_{\mch}.
 \end{equation}
Differentiate \eqref{eq:hv1} to get 
$$\partial_\tau^2 \sigma_1= \langle (\partial_\tau^2   A) \phi_1, \phi_1\rangle_{\mch}+ 2\re  \langle( \partial_\tau A )\phi_1,\partial_\tau \phi_1\rangle_{\mch}. $$
Now use \eqref{eq:ip1}  and $\re\langle \partial_\tau \phi_1, \phi_1 \rangle_{\mch}=0$
to write $\partial_\tau \phi_1$ in terms of $\{\phi_j\}$, giving the second variational formula, \eqref{eq:hv2}.
\end{proof}

We now return to the Dirichlet-to-Neumann operator, and 
consider the special case of $\lambda=i\kappa$ with $\kappa >0$.
For such  $\kappa$ the Dirichlet-to-Neumann  operator  $\dtn(i\kappa)$ is bijective
from $H^1(\partial \Omega)$ to $L^2(\partial \Omega)$: see Lemma 3.5 of \cite{gmz}. Hence (see Proposition 8.3 of Appendix A of \cite{taylor})   $\dtn(i\kappa)$ is  non-negative and self-adjoint on $L^2(\partial \Omega)$, with  domain $H^1(\partial \Omega)$ and with  
discrete spectrum accumulating at infinity.
Our theorem shows 
$\lim_{\kappa \downarrow 0}\dtn(i\kappa)$ exists, and we  we denote it $\dtn(0)$.

In fact, $\dtn(0)$ is self-adjoint and non-negative as well, with discrete spectrum accumulating only at infinity.  To see this, note that for $\kappa >0$
and $\mu \in \Complex$
\begin{equation}\label{eq:dtninverse}(\dtn(0)-\mu)(\dtn(i\kappa)-\mu)^{-1}=I+(\dtn(0)-\dtn(i\kappa))(\dtn(i\kappa)-\mu)^{-1}.
\end{equation}
Since by the first proof of Theorem \ref{t:dtn}, $\dtn(0)-\dtn(i\kappa): H^{1/2}(\partial \Omega)\rightarrow H^{1/2}(\partial \Omega)$ and $(\dtn(i\kappa)-\mu)^{-1}:L^2(\partial \Omega)\rightarrow H^1(\partial \Omega)$ (for $\mu \not \in \operatorname{spec}(\dtn(i\kappa))$),
$$(\dtn(0)-\dtn(i\kappa))(\dtn(i\kappa)-\mu)^{-1}:L^2(\partial \Omega)\rightarrow L^2(\partial \Omega)$$
is compact.  Moreover, for $\mu\ll0$ the right hand side of \eqref{eq:dtninverse} is invertible. Thus $\dtn(0)$ has compact resolvent,
and hence has discrete spectrum accumulating only at infinity.  
  A Green's theorem argument shows that $\dtn(0)$ is symmetric with one-dimensional null space 
 spanned by the constant functions.   Moreover, we see from this that $\dtn(0)+i$ is invertible, so that $\dtn(0)$ is self-adjoint.


For $\kappa \geq 0$, denote the smallest eigenvalue of $\dtn(i \kappa)$ by $\sigma_1(i \kappa)$.  Note that $\sigma_1(0)=0$, and $0$ has multiplicity
$1$ as an eigenvalue of $\dtn(0)$.  By our discussion above, $0$ is an isolated eigenvalue of $\dtn(0)$.  Then 
%
there is an $\epsilon>0$ so that $\sigma_1(i \kappa)$ is an analytic function of $\kappa \in (0,\epsilon)$ and is
continuous on $[0,\epsilon]$: see \cite[Theorem XII.8]{rs}.
Moreover, if $\phi_1(i \kappa)$ is the associated eigenfunction of $\dtn(i\kappa)$ with $\| \phi_1(i\kappa)\|_{L^2(\partial \Omega)}=1$, then
$\phi$ can be chosen to depend smoothly on $\kappa \in (0,\epsilon)$.   

\begin{prop}\label{p:dtnev}
For $\Omega$ as in Theorem \ref{t:dtn} and $\kappa >0$
\begin{equation}\label{eq:eva}
\sigma_1(i\kappa)=- \frac{2\pi}{\ell (\partial \Omega)} \frac{1}{\log (i \kappa )-a}+ O((\log \kappa )^{-2})
\end{equation}
as $\kappa \downarrow 0$.  Here
 $\ell(\partial \Omega)$ is the length of $\partial \Omega$.
\end{prop}
In dimension two this answers a question of D. Grebenkov
\cite[page 11]{gr22}.  Grebenkov computed this  when 
$\mathscr O$ is a disk (see equations (62) of \cite{gr19}, and (C.2) of \cite{gr21}) and has shown that such quantities are related to the behavior of `boundary local time'--see \cite{gr19}.

\begin{proof}
Let $\langle \cdot , \cdot \rangle_{L^2(\partial \Omega)}$ denote the inner product in $L^2(\partial \Omega)$,  and 
let  `$\cdot$' denote differentiation with respect to
$(\log (i\kappa)-a)^{-1}$. 
By \eqref{eq:hv1},
$$\dot{\sigma}_1(i\kappa)=  \langle \dot{\dtn}(i\kappa)\phi_1(i\kappa),\phi_1(i\kappa)\rangle_{L^2(\partial \Omega)},$$
for $\kappa \in (0,\epsilon)$.
Because $\phi_1$ is continuous at $0$ with $\phi_1(0) =  (\ell(\partial \Omega))^{-1/2}$ and $\dot {\dtn}(i\kappa)$ is continuous at $0$, this 
identity holds in the limit as $\kappa \downarrow 0$.
  Now we note that
we can find the unique $u\in L^2(\Omega)$ satisfying \eqref{eq:udef} with  $\lambda =i\kappa$ and $f= 1$ by setting
$u(i\kappa)= \chi -R(i\kappa) (-\Delta +\kappa^2)\chi$
where $\chi\in C_0^\infty(\Real^2)$ is $1$ in a neighborhood of $\mathscr O$.
Our asymptotics of $R(i\kappa) $ from \eqref{e:resexp} and the  formulas \eqref{e:b00b01dc2}
imply $$\chi u(i\kappa) = \chi ( \chi +1-\chi + (\log (i\kappa )-a)^{-1} G)+ O(\kappa^2 \log \kappa).$$
Thus 
$$\dtn(i\kappa)\phi_1(0)=  (\log (i\kappa )-a)^{-1} (\ell(\partial \Omega))^{-1/2} \partial_\nu G +O(\kappa^2 \log \kappa),$$
and, recalling our expansions hold under differentiation with  respect to $(\log (i\kappa)-a)^{-1}$ as well gives $\lim_{\kappa \downarrow 0}\langle \dot{\dtn}(\kappa)\phi_1(\kappa),\phi_1(\kappa)\rangle_{L^2(\partial \Omega)}= \langle \partial _\nu G , 1\rangle _{L^2(\partial \Omega)} (\ell(\partial \Omega))^{-1}$.
A Green's theorem argument then shows that $\int_{\partial\Omega}\partial_{\nu}G =  -2\pi$.  This completes the proof of \eqref{eq:eva} except that it gives  an error $o((\log \kappa)^{-1})$.

To get an error $O((\log \kappa)^{-2})$, we use \eqref{eq:hv2}.  For this, we introduce the notation that 
$\{\phi_j(i\kappa)\}$ are a complete orthonormal set of eigenfunctions of $\dtn(i\kappa)$ satisfying $\dtn(i\kappa)\phi_j(i\kappa)=\sigma_j(i\kappa)\phi_j(i\kappa)$
with $\sigma_1 \leq\sigma_2\leq \sigma_3\leq\cdots$.   
Then \eqref{eq:hv2} implies, for $\kappa>0$ sufficiently small
\begin{equation*}\label{eq:sv}
\ddot{\sigma}_1(i\kappa)= \langle \ddot{\dtn}(i\kappa)\phi_1(i\kappa),\phi_1(i\kappa)\rangle _{L^2(\partial \Omega)}+ 2 \sum_{j=2}^\infty \frac{1}{\sigma_1(i\kappa)-\sigma_j(i\kappa)}\left| \langle \phi_j(i\kappa),\dot{\dtn}(i\kappa)\phi_1(i\kappa)\rangle _{L^2(\partial \Omega)}\right|^2.
\end{equation*}
Using that the right hand side has a finite limit as $\kappa \downarrow 0$ allows us to improve the error in \eqref{eq:eva} to $O((\log \kappa)^{-2})$.
\end{proof}



\appendix

\section{Polar and nonpolar sets}\label{a:nonpolar}
In this appendix we present  some standard material about polar and nonpolar sets. Section \ref{s:basicpolar} contains the elementary basic facts used in the rest of the paper. Section \ref{s:polarbackground} has  more general background and context. Throughout, let $\mathscr O \subset \mathbb R^2$ be compact,  let $\Omega = \mathbb R^2 \setminus \mathscr{O}$, and recall that $\mathscr O $  is \textit{polar} if $C_0^\infty(\Omega)$ is dense in $H^1(\mathbb R^2)$. 

\subsection{Basic facts}\label{s:basicpolar}

We begin with a geometric necessary condition for a compact set to be polar.

\begin{lem}\label{l:nonpgeo}
Let $\mathscr O \subset \mathbb R^2$ be a compact set. If $\mathscr O$ is polar, then  the projection of $\mathscr O$ onto any line in $\mathbb R^2$ has measure zero.
\end{lem}

\begin{proof}
Without loss of generality, we are projecting onto the $x_2$ axis, and 
$x_1>0$ on $\mathscr O$. Let 
\[W = \{(x_1+t,x_2) \in \mathbb R^2 \text{ such that } (x_1,x_2) \in \mathscr O \text{ and } t \ge 0\}.\]
It is enough to show that $W$ has measure zero in $\mathbb R^2$. Let $u \in C_0^\infty(\Omega)$.  By integration by parts and Cauchy--Schwarz (Theorem 181 of \cite{hlp}), for any $s>0$ we have
\[
\|x_1^{\frac{-1-s}2}u\|_{L^2(W)}^2 =  \int_W x_1^{-1-s} |u|^2 = 
\frac 2 s \re \int_W x_1^{-s} u \partial_{x_1}\bar u \le \frac 2 s \|x_1^{\frac{-1-s}2}u\|_{L^2(W)}\|x_1^{\frac{1-s}2}\partial_{x_1}u\|_{L^2(W)},
\] 
which implies Hardy's inequality (Theorem 330 of \cite{hlp}):
\begin{equation}\label{e:hardy}
\|x_1^{\frac{-1-s}2}u\|_{L^2(W)} \le \frac 2 s \|x_1^{\frac{1-s}2}\partial_{x_1}u\|_{L^2(W)}.
\end{equation}
Since $\mathscr{O}$ is polar, by density \eqref{e:hardy} holds for all $u \in H^1(\mathbb R^2)$. Applying \eqref{e:hardy} with $s =3$ and $u(x) = e^{-|x|^2/n}$, and letting $n \to \infty$, gives $\int_W x_1^{-4} = 0$, which implies that $W$ has measure zero.
\end{proof}

Compact polar sets have many well-known  equivalent characterizations, some of which we discuss in Section \ref{s:polarbackground}. For our main results we need only Lemma \ref{l:nonpuse}, which follows Section 13.2 of \cite{maz}.

\begin{lem}\label{l:nonpuse}
Let  $\mathscr O \subset \mathbb R^2$ be compact, and let $\Omega = \mathbb R^2 \setminus \mathscr O$. The following are equivalent:
\begin{enumerate}
\item $\mathscr O$ is polar.
\item The constant function $1$ is locally in $H^1_0(\Omega)$.
\item $\inf\{\|u\|_{H^1(\mathbb R^2)}\colon u \in C_0^\infty(\mathbb R^2), \ u = 1 \text{ near } \mathscr{O} \} = 0$.
\end{enumerate}
\end{lem}

\begin{proof}

(1) $\Longrightarrow$ (2). This follows from the fact that the constant function  $1$ is locally in $H^1(\mathbb R^2)$.

(2) $\Longrightarrow$ (3). Let $v_n$ be a sequence in $C_0^\infty(\Omega)$ converging to $1$ in $H^1_{\rm loc}(\mathbb R^2)$. Fix $\chi \in C_0^\infty(\mathbb R^2)$ such that $\chi=1$ near $\mathscr O$. Then $\inf_n\{\|(1-v_n)\chi\|_{H^1(\mathbb R^2)} \} = 0$.

(3) $\Longrightarrow$ (1). Let  $u_n$ be a sequence of functions in $C_0^\infty(\mathbb R^2)$ which are $1$ near $\mathscr O$ such that $\|u_n\|_{H^1(\mathbb R^2)} \to 0$.
Let $\varphi \in C_0^\infty(\mathbb R^2)$. For any $w \in H^{-1}(\mathbb R^2)$ which vanishes on $C_0^\infty(\Omega)$, we have 
\[\langle w , \varphi \rangle = \lim_{n \to \infty} \langle w , (1-u_n)\varphi\rangle = 0.\]
By the separating hyperplane version of the Hahn--Banach theorem (see Proposition 4.6 of Appendix~A of \cite{taylor}), it follows that the $H^1(\mathbb R^2)$ distance from $\varphi$ to  $C_0^\infty(\Omega)$ is zero.
\end{proof}

\subsection{Background and context}\label{s:polarbackground}
If $\mathscr O$ is nonpolar, then $- C(\mathscr O)$ is known as Robin's constant, because it solves Robin's problem, which asks for the constant value assumed  on $\mathscr O$  by the potential of the equilibrium unit charge distribution on $\mathscr O$ \cite{robin}.
The quantity $e^{C(\mathscr O)}$ is known as the logarithmic capacity of $\mathscr O$ and it measures the size of $\mathscr O$. For example, a disk or circle has logarithmic capacity equal to its radius. See Sections V.2 and V.3 of \cite{nevan}, and Chapter 5 of \cite{ransfordbook}, for general introductions, and see  \cite{ransford,batr} for more on computing $C(\mathscr O)$. 

To expand on the above, and to also make contact with the theory of subharmonic functions, let
\[
 J_A(\mu,\nu) :=\int \!\! \int \log \tfrac A{|x-y|}\, d\mu(x)\,d\nu(y),
\]
for  $A>0$ and  $\mu, \ \nu$ finite signed 
Borel measures of compact support in $\mathbb R^2$. By Theorem~1.16 of \cite{landkof}, if $\mu \ne 0$ and the diameter of the support of $\mu$ is $\le A$, then $0 < J_A(\mu,\mu) \le \infty$.

\begin{lem}\label{l:pol2}
Let  $\mathscr O \subset \mathbb R^2$ be compact. The following are equivalent:
\begin{enumerate}
\item $\mathscr O$ is polar. 
\item $J_{1}(\mu,\mu) = +\infty$ for every nonzero finite signed Borel measure $\mu$ supported on $\mathscr O$.
\item $\mathscr O = \{x \in \mathbb R^2 \colon u(x) = -\infty\}$ for some subharmonic function $u$ on $\mathbb R^2$.
\end{enumerate}
\end{lem}

In physical terms, if $\mu$ and $\nu$ are two distributions of some finite quantity of charge, then $J_A(\mu,\mu) - J_A(\nu,\nu)$ is the difference of their electrostatic potential energies. Thus, a compact set is  polar if and only if it is so small that gathering a finite quantity of charge onto it requires infinite work. 

By Theorems 3.7.6 and 5.2.1 of \cite{ransfordbook}, if $\mathscr O$ is not polar, then $C(\mathscr O) = - \min_\mu\{J_1(\mu,\mu)\}$, where the minimum is taken over all Borel probability measures supported in $\mathscr O$. If $\mathscr O$ is the unit circle,  then $C(\mathscr O) = 0$. If $\mathscr O$ is  any nonpolar set, then $C(\mathscr O)$ is the work done moving a unit quantity of charge from its equilibrium distribution  on $\mathscr O$  to its equilibrium distribution on the unit circle.

Let us also briefly mention some further geometric characterizations and properties of polar sets. By Theorem 5.5.2 of \cite{ransfordbook}, a compact set is polar if and only if its transfinite diameter is zero, and thus Lemma \ref{l:nonpgeo} is a special case of Theorem 2 of Section VII.2 of \cite{goluzin}. Polar sets have Hausdorff dimension zero: see Section 3.2 of \cite{ransfordbook} and Section V.6 of \cite{nevan}.  By Kakutani's Theorem, \cite[Section 8.3]{MoPe}, a compact set is polar if and only if  Brownian motion hits it with probability zero.

\begin{proof}[Proof of Lemma \ref{l:pol2}]
(1) $\Longrightarrow$ (2).  It is enough to prove that, if $u \in C_0^\infty(\mathbb R^2)$  is   $1$ near $\mathscr{O}$, then
\begin{equation}\label{e:denylions}
 2 \pi (\mu(\mathscr O))^2 \le J_{A}(\mu,\mu) \int |\nabla u|^2,
\end{equation}
where $A$ is the diameter of the support of $u$.
We follow Lemma 1.1 of Chapter II of \cite{denylions}. Let $\psi = - \Delta u$. By Cauchy--Schwarz, we have
\begin{equation}\label{e:dl1}
J_A(\psi,\mu)^2\le J_{A}(\mu,\mu) J_{A}(\psi,\psi) .
\end{equation}
Since $\frac {-1} {2\pi} \Delta \log  \frac {A}{|x|} = \delta(x)$ (see Proposition 4.9 of Chapter 3 of \cite{taylor}), we have
\[
\int  \log\tfrac {A}{|x-y|} \psi(x) \, dx = 2 \pi u(y),
\]
which gives
\begin{equation}\label{e:dl2}
J_{A}(\psi,\psi)  = -2 \pi \int u \Delta u =  2 \pi \int |\nabla u|^2,
\end{equation}
and, using also the fact that $u = 1$ on $\mathscr O$,
\begin{equation}\label{e:dl3}
 J_A(\psi,\mu) = 2 \pi \int u \, d \mu = 2 \pi \int d\mu = 2 \pi \mu(\mathscr O).
\end{equation}
Inserting \eqref{e:dl2} and \eqref{e:dl3} into \eqref{e:dl1} gives \eqref{e:denylions}.

(2) $\Longrightarrow$ (1). This is a special case of equation (2.1) of \cite{wallin}.

(2) $\Longleftrightarrow$ (3). See Theorem 3.5.1 and Corollary 3.5.4 of \cite{ransfordbook}.
\end{proof}

\section{Convergence near zero of series with logarithmic terms}\label{a:series}

To analyze series  of the form $\sum_{j,k} c_{j,k}(\log \lambda - a)^k \lambda^j$, where $\lambda \in \Lambda$, it is convenient to introduce a new coordinate $\nu$ on $\Lambda$, defined by
\[
 \nu = \lambda e^{-a}, \qquad \log \nu = \log \lambda - a.
\]
Then
\[
 \sum_{j,k} c_{j,k} (\log \lambda - a)^k \lambda^j =  \sum_{j,k} c_{j,k} e^{aj} (\log \nu)^k \nu^j.
\]
In our applications, $\im a = \pi/2$, so the physical region $0 < \arg \lambda < \pi$ corresponds to $-\pi/2 < \arg \nu < \pi/2$.

\begin{lem}\label{l:appser}
Let $N \in \mathbb N_0$, let $(\mathcal B, \| \cdot\|)$  be a Banach space, and for each $j\in \mathbb N$ and each $k\le Nj$, let $V_{j,k}$ be an element of $\mathcal B$. Let $\nu_0 \in (0,1)$, and suppose the series
\[
 \sum_{j=1}^\infty \sum_{k = -\infty}^{Nj} \|V_{j,k}\| |\log \nu_0|^k \nu_0^j.
\]
converges. Then for any $\varphi>0$, there is $\nu_1>0$ such that the series
\begin{equation}\label{e:vseries}
\sum_{j=1}^\infty \sum_{k = -\infty}^{Nj} V_{j,k} (\log \nu)^k \nu^j,
\end{equation}
and 
\begin{equation}\label{e:v'series}
 \sum_{j=2}^\infty \sum_{k = -\infty}^{Nj} V_{j,k} \big( k  + j  \log \nu\big) (\log \nu)^{k-1} \nu^{j-1}.
\end{equation} 
converge absolutely in $\mathcal B$, uniformly on $ U_{\nu_1,\varphi}:=\{\nu \in \Lambda \colon 0< |\nu| < \nu_1 \text{ and } |\arg \nu| < \varphi\}.$
\end{lem}
Since the terms of \eqref{e:vseries} are holomorphic, it follows that for any $\varphi>0$, there is $\nu_1>0$ such that the function  $f(\nu)$ given by \eqref{e:vseries}
is holomorphic on $U_{\nu_1,\varphi}$, with $f'(\nu)$ given by term-by-term differentiation of \eqref{e:vseries} (see Theorem 1 of Chapter 5 of \cite{ahlfors}). Moreover, since the series for $f'$ is of the same form as the series for $f$, by induction these series may be freely differentiated term by term. Note, however, that  convergence is uniform on $U_{\nu_1,\varphi}$ only if we omit terms whose norm goes to $\infty$ as $|\nu| \to 0$, and the value of $\nu_1$ depends on the number of differentiations as well as on $\varphi$.

\begin{proof}
For \eqref{e:vseries}, observe that for $\nu \in U_{\nu_1,\varphi}$ we have
\begin{equation}\label{e:lnnuknu}
 |\log \nu|^N |\nu|  \le |\log |\nu| + i \varphi|^N |\nu| \le |\log \nu_1 + i \varphi|^N \nu_1 \le |\log \nu_0|^N \nu_0, 
\end{equation}
when $\nu_1>0$ is small enough. Combining \eqref{e:lnnuknu} with the fact that $|\log \nu|^{k-Nj} \le |\log \nu_0|^{k-Nj}$ when $k \le Nj$, and with $ |(\log \nu)^k \nu^j| =    |\log \nu|^{k-Nj} \big(|\log \nu|^N |\nu|\big)^j $, we obtain
\[\begin{split}
 \sum_{j=1}^\infty  \sum_{k = -\infty}^{Nj} \|V_{j,k}  (\log \nu)^k \nu^j\| \le \sum_{j=1}^\infty  \sum_{k = -\infty}^{Nj} \|V_{j,k}\|  |\log \nu_0|^k |\nu_0|^j,
\end{split}\]
which implies the absolute convergence of the series \eqref{e:vseries}, uniformly on $U_{\nu_1,\varphi}$. 

For \eqref{e:v'series}, it is enough to show that there exist positive constants $\nu_1$  and $C$ such that
\begin{equation}\label{e:kjlognu}
 \big| k  + j  \log \nu\big| |\log \nu|^{k-1} |\nu|^{j-1} \le C |\log \nu_0|^k \nu_0^j,
\end{equation}
whenever $\nu \in U_{\nu_1,\varphi}$,  and $ k \le N j$. 

The case $k<0$ is easier. It suffices to take $\nu_1$ small enough that we have $|\log \nu / \log \nu_0|^k \le 2^{k}$ and $|\nu/\nu_0|^{j-1} \le 2^{1-j}$, and then put $C = \sup \{| k  + j  \log \nu| |\log \nu|^{-1} 2^k 2^{1-j} \colon j \ge 2, \, k < 0, \, \nu \in U_{\nu_1,\varphi}\}$.

For the case $k \ge 0$, observe that since $|k + j \log \nu| \le j(N + |\log \nu|) \le 2 j |\log \nu|$ when $\nu_1 \le e^{-N}$, it is enough to obtain
\[
2j|\log \nu|^k |\nu|^{j-1} \le C |\log \nu_0|^k \nu_0^j.
\]
Analogously to \eqref{e:lnnuknu}, we have $(|\log \nu|^{M} |\nu|)^{j-1} \le (|\log \nu_0|^{M} \nu_0)^{j-1}$ for all $\nu \in U$ when $\nu_1>0$ is small enough, with $M \in \mathbb N$ to be determined later, and so it is enough to obtain
\[
2j|\log \nu|^{k-M(j-1)}   \le C\nu_0 |\log \nu_0|^{k-M(j-1)}.
\]
Since $k - Nj \le 0$, we have $|\log \nu|^{k-Nj} \le |\log \nu_0|^{k-Nj}$, and so it is enough to obtain
\[
2j|\log \nu|^{Nj - M(j-1)}   \le C\nu_0|\log \nu_0|^{Nj-M(j-1)}.
\]
We now take $M = 2N+1$, so that the exponent $Nj - M(j-1) = 2N+1 - (N+1)j$ is negative and decreasing for $j \ge 2$. We further require  $\nu_1 \le \nu_0^2$, so that $2|\log \nu|^{-1} \le  |\log \nu_0|^{-1}$. It is then enough to obtain
\[
j2^{(N+1)(2-j)}   \le C\nu_0,
\]
which we do by putting $C = \max \{j2^{(N+1)(2-j)} /\nu_0 \colon j \ge 2\} = 2/\nu_0$.
\end{proof}

\section{Scattering by a disk}\label{a:disk}

Let $\mathscr{O} = \{x \in \mathbb R^2 \colon |x| \le \rho\}$ for some $\rho > 0$. Solutions to the Helmholtz equation $(-\Delta - \lambda^2)u = 0$ can be written in polar coordinates in terms of the Hankel functions (see \cite[Section 7.4.1]{olver}) as
\[
 u(r,\theta) = \sum_{\ell = -\infty}^\infty \Big(a_\ell H_\ell^{(1)}(\lambda r) + b_\ell H_\ell^{(2)}(\lambda r)\Big) e^{i\ell \theta},
\]
and they vanish at $\partial \mathscr O$ if and only if
\[
 a_\ell H_\ell^{(1)}(\lambda \rho) + b_\ell H_\ell^{(2)}(\lambda \rho) = 0, \qquad \text{for every } \ell.
\]
Consequently, with respect to the basis $\{e^{i\ell\theta}\}_{\ell \in \mathbb Z},$ the scattering matrix  $S(\lambda)$ is the diagonal matrix mapping the incoming data (the $b_\ell)$ to the outgoing data (the $a_\ell$), i.e. its $\ell$-th entry is given by $-H^{(2)}_\ell(\lambda \rho)/H^{(1)}_\ell(\lambda \rho)$.
Hence,
\begin{equation}\label{e:slh1h2}
 \sigma(\lambda) = \frac 1 {2\pi i} \log \det S(\lambda) = \frac 1 {2\pi i} \sum_{\ell = -\infty}^\infty \log \Big(- \frac {H^{(2)}_\ell(\lambda \rho)}{H^{(1)}_\ell(\lambda \rho)}\Big).
\end{equation}

\noindent\textbf{Acknowledgments.} 
It is a pleasure to thank Maciej Zworski for many helpful discussions, including suggesting the application to scattering phase asymptotics and observing that the constant $a$ appearing in our expansions has a natural interpretation in terms of the logarithmic capacity. 
The authors thank Tom ter Elst, Hamid Hezari, and Steve Hofmann for helpful conversations, and
gratefully acknowledge the partial support of the Simons Foundation (TC, collaboration grant for mathematicians) and the National Science Foundation (KD, Grant DMS-1708511).

\end{document}